\documentclass[12pt, oneside,reqno]{amsart}

\usepackage{amsmath,amssymb,amsthm,textcomp,mathtools}
\usepackage{amsfonts,graphicx}
\usepackage[mathscr]{eucal}
\usepackage{color}
\usepackage{csquotes}
\usepackage{enumitem}
\usepackage{setspace}
\numberwithin{equation}{section}

\theoremstyle{definition}

\numberwithin{equation}{section}

\date{\today}

\newtheorem{theorem}{\bf Theorem}[section]
\newtheorem{remark}{\bf Remark}[section]
\newtheorem{proposition}{Proposition}[section]
\newtheorem{lemma}{Lemma}[section]

\newtheorem{example}{Example}[section]

\newtheoremstyle
{remarkstyle}
{}
{11pt}
{}
{}
{\bfseries}
{:}
{     }
{\thmname{#1} \thmnumber{#2} }

\theoremstyle{remarkstyle}

\begin{document}
	\title{Some probabilistic properties and time-changed versions of a renewal process based on Mittag-Leffler waiting times}
	\author[Mostafizar Khandakar]{Mostafizar Khandakar}
	\address{Mostafizar Khandakar, Department of Science and Mathematics, Indian Institute of Information Technology Guwahati, Bongora, Assam 781015, India.}
	\email{mostafizar@iiitg.ac.in}
	\author[Bratati Pal]{Bratati pal}
	\address{Bratati pal, Department of Science and Mathematics, Indian Institute of Information Technology Guwahati, Bongora, Assam 781015, India.}
	\email{bratati.pal@iiitg.ac.in}
	
	\date{\today}
	\begin{abstract}
		In this paper, we obtain some additional probabilistic properties of the renewal process $\{\hat{N}_{\alpha}(t)\}_{t\ge0}$, $0<\alpha\le 1$ introduced by Beghin and Orsingher (2010). A  time-changed relationship connecting $\{\hat{N}_{\alpha}(t)\}_{t\ge0}$ with its special case $\{\hat{N}(t)\}_{t\ge0}$ by means of the random time process $\{T_{2\alpha}(t)\}_{t>0}$ whose distribution is related to a fractional diffusion equation is established. We compute its various distributional properties such as the 
		variance, factorial moments, moment generating function, moments, covariance in the Laplace domain, {\it etc.} We show that the ratios given by $\{\hat{N}_{\alpha}(t)\}_{ t \ge 0}$ and its power over their means tend to $1$ in probability. Moreover, we derive an integral form of its  bivariate distribution and describe the scaling limits of its marginal distributions. It is also shown that its one-dimensional distributions are not infinitely divisible. Furthermore, we study the compound version of $\{\hat{N}_{\alpha}(t)\}_{ t \ge 0}$ and discuss an application to ruin theory. Later, we consider two time-changed versions of $\{\hat{N}_{\alpha}(t)\}_{ t \ge 0}$ which are obtained by time-changing it with an independent 
		L\'evy subordinator and its inverse. Some distributional properties and examples are discussed for these time-changed processes. 
	\end{abstract}
	
	\maketitle
	
	\section{Introduction}
	The Poisson process $\{N (t)\}_{t\ge0}$ with intensity $\lambda >0$ is a L\'evy process as well as a renewal process with exponentially distributed interarrival times. The study of fractional generalizations
	of the Poisson process has attracted a considerable attention over the past two decades. The time fractional Poisson process (TFPP) and the space fractional Poisson process (see Orsingher and Polito (2012)) are the two main fractional version of the Poisson process. We denote the TFPP by $\{N_\alpha(t)\}_{t \ge 0}$, $0<\alpha\le 1$. It is a renewal process with Mittag-Leffler distributed waiting times (see Mainardi {\it et al.} (2004)) and its  state probabilities $p_\alpha(n,t)=\mathrm{Pr}\{N_\alpha(t)=n\}$ satisfy the following (see Laskin (2003), Beghin and Orsingher (2009)): 
	\begin{equation*}
		\frac{\mathrm{d}^{\alpha} }{\mathrm{d}t^{\alpha}} p_\alpha(n,t)= -\lambda (p_\alpha(n,t) - p_\alpha(n-1,t)), \ n \ge 0,
		\label{eq:frac_poisson}
	\end{equation*}
	with $p_\alpha(-1,0)=0$ and initial conditions $p_\alpha(0,0)=1$, and $p_\alpha(n,0)=0$, $n\ge 1$. Here,
	$\dfrac{\mathrm{d}^{\alpha} }{\mathrm{d}t^{\alpha}}$
	is the Caputo fractional derivative defined in (\ref{caputo}). Meerschaert {\it et al.} (2011) showed that the Poisson process time-changed  by an independent
	inverse stable subordinator is also a TFPP. For other generalizations of the Poisson process, we refer the reader to Beghin (2012), Di Crescenzo {\it et al.} (2016), Polito and Scalas (2016), Beghin {\it et al.} (2024), {\it etc.}, and the references therein.
	
		It is important to note that the time-changed point processes have potential applications in different fields such as finance, hydrology,
	econometrics, {\it etc}. These widely studied time-changed stochastic processes are usually
	constructed by time-changing a point process with an independent subordinator and
	its inverse. A subordinator $\{D_f(t)\}_{t \ge 0}$ is a one-dimensional L\'evy process with non-decreasing sample paths satisfying $D_f(0)=0$ almost surely (a.s.). It is characterized by the following Laplace 
	transform (see Applebaum (2009), Section 1.3.2):
	\begin{equation}\label{levy}
		\mathbb{E}\big(e^{-sD_f(t)}\big)= e^{-tf(s)},\ \ s>0,
	\end{equation}
	where
	\begin{equation*}
		f(s) = cs + \int_0^{\infty} (1 - e^{-sx}) \bar{\nu}(\mathrm{d}x),  
	\end{equation*}
	is called the Bern\v{s}tein function and $c\geq 0$ is the the drift coefficient. Here, $ \bar{\nu}(\cdot)$ is a non-negative L\'evy measure on $(0,\infty)$ that satisfies $  \displaystyle\int_0^{\infty} (x \wedge 1) \bar{\nu}(\mathrm{d}x) < \infty$. The first passage time of a  subordinator, that is, 
	\begin{equation*}
		Y_f(t)\coloneqq \inf \{x \geq 0 : D_f(x) > t\}, \ \  t \geq 0,
	\end{equation*}
	is called the inverse subordinator.
	
	Orsingher and Toaldo (2015) constructed a class of point processes by time-changing the Poisson process with an independent L\'evy subordinator, that is, the process $\{N (D_f(t))\}_{t\ge0}$ which performs integer valued jumps of arbitrary height. Recently, Maheshwari and Vellaisamy (2019) introduced and studied two time-changed versions of the TFPP using a L\'evy subordinator and its inverse as a time-changed component. For other time-changed stochastic processes, we refer the reader to Aletti {\it et al.} (2018), Kataria and Khandakar (2022), Kataria {\it et al.} (2025), {\it etc.}, and the references therein.

	Beghin and  Orsingher (2010) introduced and studied a renewal process $\{\hat{N}_{\alpha}(t)\}_{t\ge0}$, $0<\alpha\le 1$, whose state probabilities $\hat{p}_{\alpha}(n,t)=\mathrm{Pr}\{\hat{N}_{\alpha}(t)=n\}$ satisfy the following system of fractional differential equations: 
	\begin{equation}\label{diff1}
		\frac{\mathrm{d}^{2\alpha}}{\mathrm{d}t^{2\alpha}}\hat{p}_{\alpha}(n,t)+2\lambda\frac{\mathrm{d}^{\alpha}}{\mathrm{d}t^{\alpha}}\hat{p}_{\alpha}(n,t)=-\lambda^{2}(\hat{p}_{\alpha}(n,t)-\hat{p}_{\alpha}(n-1,t)),
	\end{equation}
	with the initial conditions $\hat{p}_{\alpha}(n,0)=\begin{cases*}             1, \ \ n=0,\\             0, \ \ n\ge 1,            \end{cases*}$$ $ for $0<\alpha\le 1$ and $\hat{p}_{\alpha}^{\prime}(n,0)=0$, $n \ge 0$, for $\frac{1}{2}<\alpha\le 1$, and $\hat{p}_{\alpha}(-1,t)=0$. They obtained its probability mass function (pmf), probability generating function (pgf), mean, asymptotic behavior of its interarrival-time density, {\it etc}. It is connected to the TFPP by the following relationship: 
	\begin{equation*}
		\hat{p}_{\alpha}(n,t)=p_{\alpha}(2n,t)+p_{\alpha}(2n+1,t),\ n\ge0,
	\end{equation*}
	that is, the process $\{\hat{N}_{\alpha}(t)\}_{t\ge0}$ can be viewed as the TFPP which jumps upward at even-order events $A_{2n}$ and the
	probability of the successive odd-indexed events $A_{2n+1}$ is added to that of $A_{2n}$.

	Sakhno and Storozhuk (2026) obtained the following subordination relationship for $\{\hat{N}_{\alpha}(t)\}_{t\ge0}$:
	\begin{equation*}
		\hat{N}_{\alpha}(t)\overset{d}{=}N(L_\alpha(t)),\ \ t>0,
	\end{equation*}
	where the Poisson process $\{N(t)\}_{t\ge0}$ is independent of $L_\alpha(t)$. Here, $\overset{d}{=}$ denotes the equality in distribution and
	\begin{equation*}
		L_\alpha(t)=\inf\{s\ge 0: D^1_{2\alpha}(s)+(2\lambda)^{1/\alpha}D^2_{\alpha}(s)\ge t\}, \ \ 0<\alpha\le \frac{1}{2}, \ \lambda>0,
	\end{equation*} 
where $D^1_{2\alpha}(t)$ and $D^2_{\alpha}(t)$ are two independent stable subordinators. For more details on the process $L_\alpha(t)$, we refer the reader to D'Ovidio {\it et al.} (2014).  Recently, Beghin {\it et al.} (2026) introduced a generalization of $\{\hat{N}_{\alpha}(t)\}_{t\ge0}$ by replacing the Caputo derivative in \eqref{diff1} by a stretched non-local operator. 
	
	For $\alpha = 1$, the process $\{\hat{N}_{\alpha}(t)\}_{t\ge0}$ reduces to $\{\hat{N}(t)\}_{t\ge0}$ whose state probabilities $\hat{p}(n,t)=\mathrm{Pr}\{\hat{N}(t)=n\}$ satisfy
	\begin{equation*}\label{pdiff}
		\frac{\mathrm{d}^{2}}{\mathrm{d}t^{2}}\hat{p}(n,t)+2\lambda\frac{\mathrm{d}}{\mathrm{d}t}\hat{p}(n,t)=-\lambda^{2}(\hat{p}(n,t)-\hat{p}(n-1,t)),\ n\ge0.
	\end{equation*}
	Note that the process $\{\hat{N}(t)\}_{t\ge0}$ can be regarded as a standard Poisson  process with Gamma $(\lambda,2)$ distributed interarrival times. It presents a generalization of the classical Poisson process and it is widely applied in modeling random motions at finite velocities. They also extend this processes by taking fractional time derivatives of order $\alpha, 2\alpha, \dots, n\alpha$. 
	
	In this paper, we study several additional probabilistic properties of the process $\{\hat{N}_{\alpha}(t)\}_{t \ge 0}$, along with its compound and time-changed versions.  In Section \ref{section2}, we present some preliminary results and definitions that will be used throughout the paper.
	
	In Section \ref{section3}, first, we show that the process $\{\hat{N}_{\alpha}(t)\}_{t \ge 0}$ and the renewal process studied by   Cahoy and Polito (2012b) are equal in distribution. Then, we establish the following subordination relationship between  the process $\{\hat{N}_{\alpha}(t)\}_{ t \ge 0}$ and  $\{\hat{N}(t)\}_{ t \ge 0}$:
	\begin{equation*}
		\hat{N}_{\alpha}(t)\overset{d}{=}\hat{N}(T_{2\alpha}(t)),\ \  t > 0,
	\end{equation*}
	where $\{T_{2\alpha}(t)\}_{t >0}$ is a suitable random time process  independent of $\{\hat{N}(t)\}_{ t \ge 0}$ whose distribution solves a fractional diffusion equation of order $2\alpha$ given in \eqref{diff}. Using the above time-changed relationship, we alternatively obtain the pgf of the process $\{\hat{N}_{\alpha}(t)\}_{ t \ge 0}$. Then an alternative expression for the pmf of the process $\{\hat{N}_{\alpha}(t)\}_{ t \ge 0}$ is derived in terms of the derivatives of the Mittag-Leffler function. Also, we compute the factorial moments, variance, covariance in the Laplace domain, moment generating function (mgf), moments, {\it etc.} of $\{\hat{N}_{\alpha}(t)\}_{ t \ge 0}$. It is shown that its one-dimensional distributions  are not infinitely divisible. An another time-changed relationship 
	$$
	\hat{N}(t) \stackrel{d}{=} \hat{N}_\alpha(H^\alpha(t)), \ \  t> 0,
	$$
	is established where $\{H^{\alpha}(t)\}_{t \ge 0}$ is random time process whose one-dimensional distribution is given in terms of $H$-function. We derive an integral form of the  bivariate distribution $\mathrm{Pr}\{\hat{N}_{\alpha}(s)=l,\hat{N}_{\alpha}(t)=m\}$ for $s \le t$, $l \le m$ of $\{\hat{N}_{\alpha}(t)\}_{ t \ge 0}$. Moreover, we show that both the ratios given by the process $\{\hat{N}_{\alpha}(t)\}_{ t \ge 0}$ and the power of the process $\{\hat{N}_{\alpha}(t)\}_{ t \ge 0}$ over their means converges to $1$ in probability. This result is relevant for applications as it means that the distance between the distribution of the process at time $t$ and its equilibrium measure remains close to $1$ up to a deterministic cutoff time, and then is close to $0$ shortly after. Also, the scaling limits of the marginal distributions of $\{\hat{N}_{\alpha}(t)\}_{ t \ge 0}$  is described. 
	
	In Section \ref{section4}, we study a compound version of the process $\{\hat{N}_{\alpha}(t)\}_{ t \ge 0}$ defined as 
	\begin{equation*}
		\hat{Z}_{\alpha}(t) \coloneqq \sum_{i=1}^{\hat{N}_{\alpha}(t)} X_i, \  \ \text{$ t \ge 0$},
	\end{equation*}
	where $\{X_i\}_{i \ge 1}$ is a sequence of positive integer valued independent and identically distributed (i.i.d.) random variables with finite mean and variance, and these are independent of $\{{\hat{N}_\alpha(t)}\}_{t\ge 0}$. The system of differential equations that governs the state probabilities of $\{\hat{Z}_{\alpha}(t)\}_{t\ge0}$ is derived. We obtain  its mean, variance and mgf. By identifying the process $\{\hat{Z}_{\alpha}(t)\}_{t\ge 0}$ as the claim process, we consider a risk process. The proposed risk model can be viewed as a particular case of the fractional insurance model introduced and studied by Beghin and Macci (2013). For this risk model, we obtain an expression of the covariance in terms of the covariance of $\{{\hat{N}_\alpha(t)}\}_{t\ge 0}$. Moreover, when the claim sizes follow an exponential distribution, we derive the probability density function (pdf) of the ruin time of the proposed risk model. Furthermore, we also discuss an asymptotic behavior of its finite time ruin probability for subexponentially distributed claim sizes when the initial capital is arbitrarily large.
	
	In Section \ref{sec5}, we explore the processes $\{\hat{N}_{\alpha}(t)\}_{t \ge 0}$ and $\{\hat{N}(t)\}_{t \ge 0}$ time-changed by an independent L\'evy
	subordinator and its inverse with finite moments of any order. The process $\{\hat{N}_\alpha(t)\}_{t \geq 0}$ time-changed  by an independent L\'evy subordinator $\{D_{f}(t)\}_{t\ge0}$ is defined by 
	$\hat{N}_\alpha^{f}(t):=\hat{N}_\alpha(D_{f}(t))$, $t\ge0$. For $\alpha=1$, it reduces to $\hat{N}^{f}(t):=\hat{N}(D_{f}(t))$, $t\ge0$. We compute the pmf, pgf, mean, variance of  $\{\hat{N}_\alpha^{f}(t)\}_{t\ge0}$ and established a version of law of iterated logarithm for it. Also, we obtain the density function of the first passage times of $\{\hat{N}^{f}(t)\}_{t\ge0}$. Furthermore, some examples are discussed by considering specific L\'evy subordinators such as the gamma 
	subordinator, tempered stable subordinator and inverse Gaussian 
	subordinator. Later, we consider the processes $\{\hat{N}_{\alpha}(t)\}_{t \ge 0}$ and $\{\hat{N}(t)\}_{t \ge 0}$ time-changed by an independent inverse
	subordinator. We have shown that these processes are renewal and identified their waiting time distribution. We derive their bivariate distributions which generalizes the bivariate distribution of $\{\hat{N}_{\alpha}(t)\}_{t \ge 0}$.

	\section{Preliminaries}\label{section2}
	In this section, we give some known results related to Mittag-Leffler function, Wright function, stable and inverse stable subordinator, some fractional derivatives, the process $\{\hat{N}(t)\}_{t\ge 0}$ and its fractional version $\{\hat{N}_{\alpha}(t)\}_{t\ge 0}$.
	\subsection{Mittag-Leffler function and Wright function}
	The three-parameter Mittag-Leffler function is defined as (see Kilbas {\it et al.} (2006), p. 45)
	\begin{equation*}
		E_{\alpha,\beta}^{\gamma}(x)\coloneqq\frac{1}{\Gamma(\gamma)}\sum_{j=0}^{\infty} \frac{\Gamma(j+\gamma)x^{j}}{j!\Gamma(j\alpha+\beta)},\ \ x\in\mathbb{R},
	\end{equation*}
	where $\alpha>0$, $\beta>0$ and $\gamma>0$. For $ \gamma =1$, it reduces to the two parameter Mittag-Leffler function and for $\beta=\gamma =1$ it reduces to the Mittag-Leffler function. Also, for $ \alpha=\beta=\gamma =1$, it reduces to the exponential function. 
	
	The following Laplace transform  hold (see Kilbas {\it et al.} (2006), Eq. (1.9.13)):
	\begin{equation}\label{lapin}
		\mathcal{L}\big(t^{\beta-1}E^{\gamma}_{\alpha,\beta}(wt^{\alpha});s\big)=\frac{s^{\alpha\gamma-\beta}}{(s^{\alpha}-w)^{\gamma}},\ \ \text{$s > 0$},
	\end{equation}
	where $w \in \mathbb{R}$ and $|ws^{-\alpha}|<1$.
	
	For large $t$, the following asymptotic result holds (see Beghin (2012), Eq. (2.44)):
	
	\begin{equation}\label{asymittg}
		E^{\gamma}_{\alpha,\beta}(-\lambda t^{\alpha})
		\sim
		\frac{\lambda^{-\gamma} t^{-\alpha\gamma}}
		{\Gamma(\beta-\alpha\gamma)},
		\ \  \beta \neq \alpha\gamma .
	\end{equation}
	Its Mellin-Barnes representation is given by (see Kilbas {\it et al.} (2006), Eq. (1.9.10))
	\begin{equation}\label{cont}
		E_{\alpha, \beta}^{\gamma}(x) = \frac{1}{2\pi i \Gamma(\gamma)} \int_{c-i\infty}^{c+i\infty} \frac{\Gamma(z)\Gamma(\gamma - z)}{\Gamma(\beta - \alpha z)} (-x)^{-z} \mathrm{d}z, \ \ x \neq 0,
	\end{equation}
	where $i=\sqrt{-1}$.  Thus, the following holds true:
	\begin{equation}\label{mbe}
		e^{x}
		= \frac{1}{2\pi i}
		\int_{c-i\infty}^{c+i\infty}
		\Gamma(z)\,(-x)^{-z}\,\mathrm{d}z.
	\end{equation}
	The following recurrence relation will be used (see Mathai and Haubold (2008), Eq. (2.3.8)):
	\begin{equation}\label{recurmittag}
		x E_{\alpha,\beta}^{\gamma}(x)=  E_{\alpha,\beta-\alpha}^{\gamma}(x)
		- E_{\alpha,\beta-\alpha}^{\gamma-1}(x).
	\end{equation}
	The Wright function is defined by (see Kilbas {\it et al.} (2006), p. 54)
	\begin{equation}\label{w}
		W_{\beta,\omega}(x)=\sum_{k=0}^{\infty}\frac{x^{k}}{k!\Gamma(k\beta+\omega)},\ \ \beta>-1,\ \omega >0,\ x\in \mathbb{R}.
	\end{equation}
	The following representations in terms of integrals of the Hankel contour $Ha$ of the Wright and  Mittag-Leffler functions will be used (see Beghin and Orsingher (2009), Remark 2.2):
	\begin{align}
		W_{-\alpha,1-\alpha}(x)& = \frac{1}{2\pi i} \int_{Ha} \frac{e^{v + x v^{\alpha}}}{v^{1-\alpha}} \mathrm{d}v,\ \  0 < \alpha \le 1, \label{con1}\\
		E_{\alpha, 1}(x) &= \frac{1}{2\pi i} \int_{Ha} \frac{e^v v^{\alpha-1}}{v^\alpha- x} \mathrm{d}v \label{con2}.
	\end{align}
	
	\subsection{Stable and inverse stable subordinator}
	A stable subordinator $\{D_{\alpha}(t)\}_{t\ge0}$, $0<\alpha<1$, is a non-decreasing  L\'evy process with  Bern\v{s}tein function $f(s)=s^\alpha$.
	Its first passage time, that is, $Y_{\alpha}(t)\coloneqq \inf\{x\ge0:D_{\alpha}(x)>t\}$, $t\ge0$ is called the inverse stable subordinator. 
	Its Laplace transform is given by
	\begin{equation*}
		\mathbb{E}\big(e^{-sY_{\alpha}(t)}\big)=E_{\alpha,1}({-st^{\alpha}}).
	\end{equation*}
	Note that it is a self-similar process (see Meerschaert and Scheffler (2004)), that is,
	\begin{equation}\label{selfsimi}
		Y_{\alpha}(t)\overset{d}{=}t^{\alpha}Y_{\alpha}(1).
	\end{equation} 
	
	\subsection{Some fractional derivatives}For $\gamma\geq 0$, the Riemann-Liouville fractional derivative is defined as (see Kilbas {\it et al.} (2006))
	\begin{equation}\label{RLd}
		\mathbb{D}_t^{\gamma}w(t):=\left\{
		\begin{array}{ll}
			\dfrac{1}{\Gamma{(m-\gamma)}}\displaystyle \frac{\mathrm{d}^{m}}{\mathrm{d}t^{m}}\int^t_{0} \frac{w(s)}{(t-s)^{\gamma+1-m}}\,\mathrm{d}s,\ \ m-1<\gamma<m,\\\\
			\displaystyle\frac{\mathrm{d}^{m}}{\mathrm{d}t^{m}}w(t),\ \ \gamma=m,
		\end{array}
		\right.
	\end{equation}
	where $m$ is a positive integer.
	
	The Caputo fractional derivative  $\dfrac{\mathrm{d}^{\alpha}}{\mathrm{d}t^{\alpha}}$ is defined as (see Kilbas {\it et al.} (2006))
	\begin{equation}\label{caputo}
		\frac{\mathrm{d}^{\alpha}}{\mathrm{d}t^{\alpha}}f(t):=\left\{
		\begin{array}{ll}
			\dfrac{1}{\Gamma{(1-\alpha)}}\displaystyle\int^t_{0} (t-s)^{-\alpha}f'(s)\,\mathrm{d}s,\ \ 0<\alpha<1,\\\\
			f'(t),\ \ \alpha=1.
		\end{array}
		\right.
	\end{equation}
	Its Laplace transform is given by (see Kilbas {\it et al.} (2006), Eq. (5.3.3))
	\begin{equation}\label{lc}
		\mathcal{L}\Big(\frac{\mathrm{d}^{\alpha}}{\mathrm{d}t^{\alpha}}f(t);s\Big)=s^{\alpha}\tilde{f}(s)-s^{\alpha-1}f(0),\ \ s>0,
	\end{equation}
	where $\tilde{f}(s)=\displaystyle\int_{0}^\infty e^{-st} f(t) \mathrm{d}t$ is the Laplace transform of $f(t)$.
	\subsection{ Some distributional  properties of the processes $\{\hat{N}_{\alpha}(t)\}_{t\ge0}$ and $\{\hat{N}(t)\}_{t\ge0}$ } Here, we give some known results of $\{\hat{N}_{\alpha}(t)\}_{t\ge0}$ and  $\{\hat{N}(t)\}_{t\ge0}$ (see Beghin and Orsingher (2010)). 
	
	The pmf of $\{\hat{N}_{\alpha}(t)\}_{t\ge0}$ is given by
	\begin{equation}\label{fracpmf}
		\hat{p}_{\alpha}(n,t)=\lambda^{2n}t^{2n\alpha}E_{\alpha,2n\alpha+1}^{2n+1}(-\lambda t^{\alpha})+\lambda^{2n+1}t^{(2n+1)\alpha}E_{\alpha,(2n+1)\alpha+1}^{2n+2}(-\lambda t^{\alpha}),\ \ n\ge0.
	\end{equation}
	Its pgf $\hat{G}_{\alpha}(u,t)=\mathbb{E}\big(u^{\hat{N}_{\alpha}(t)}\big)$ is given by
	\begin{equation}\label{pgf}
		\hat{G}_{\alpha}(u,t)=\frac{\sqrt{u}+1}{2\sqrt{u}}E_{\alpha,1}(-\lambda (1-\sqrt{u})t^{\alpha})+\frac{\sqrt{u}-1}{2\sqrt{u}}E_{\alpha,1}(-\lambda (1+\sqrt{u})t^{\alpha}),\ \ |u|\le 1.
	\end{equation}
	Its mean is  
	\begin{equation}\label{mean}
		\mathbb{E}(\Hat{N}_\alpha(t))=\lambda^2 t^{2\alpha} E_{\alpha,2\alpha+1}(-2\lambda t^{\alpha}).
	\end{equation}
	The process $\{\hat{N}_{\alpha}(t)\}_{t\ge0}$ is a renewal process whose density of the interarrival time between successive events is given by
	\begin{equation}\label{renpmf}
		f_{\hat{T}_1^{\alpha}}(t)=\mathrm{Pr}\{\hat{T}_1^{\alpha}\in\mathrm{d}t\}/ {\mathrm{d}t}= \lambda^2 t^{2\alpha-1}E^{2}_{\alpha,2\alpha }(-\lambda t^{\alpha}),
	\end{equation}
	and the Laplace transform of the interarrival time density is
	\begin{equation}\label{renlap}
		\mathcal{L}(f_{\hat{T}_1^{\alpha}}(t);s)=\tilde{f}_{\hat{T}_1^{\alpha}}(s)=\frac{\lambda^2}{(s^{\alpha} + \lambda)^2}.
	\end{equation}
	Also, its survival probability is given by 
	\begin{equation}\label{sur}
		\mathrm{Pr}\{\hat{T}_1^{\alpha} > t\}=E_{\alpha,1}(-\lambda t^\alpha) + \lambda t^\alpha E^2_{\alpha,\alpha+1}(-\lambda t^\alpha).
	\end{equation}
	For $\alpha=1$,  the process $\{\hat{N}_{\alpha}(t)\}_{t\ge0}$ reduces to $\{\hat{N}(t)\}_{t\ge0}$. Its pmf and pgf are  given by
	\begin{align}
		& \hat{p}(n,t)=\frac{(\lambda t)^{2n}}{(2n)!}e^{-\lambda t}+\frac{(\lambda t)^{2n+1}}{(2n+1)!}e^{-\lambda t},\ \ n\ge0, \label{pmf2pp}\\
		& \hat{G}(u,t)=\frac{\sqrt{u}+1}{2\sqrt{u}}e^{-\lambda (1-\sqrt{u})t}+\frac{\sqrt{u}-1}{2\sqrt{u}}e^{-\lambda (1+\sqrt{u})t}, \ \ |u|\le 1. \label{pgf2pp}
	\end{align}

	\section{Probabilistic properties of the process $\{\hat{N}_{\alpha}(t)\}_{t\ge0}$}\label{section3}
	In this section, we obtain some additional probabilistic results for the process $\{\hat{N}_{\alpha}(t)\}_{t\ge0}$. 
	
	Cahoy and Polito (2012b) considered a generalization of TFPP, denoted by $\{N^{\alpha, \delta}(t)\}_{t \ge 0}$, $\ \delta \in \mathbb{R}$ whose interarrival time density is given by 
	\[
	f^{\alpha,\delta}(t) = \lambda^\delta t^{\delta \alpha - 1}
	E_{\alpha,\delta \alpha}^{\delta}(-\lambda t^\alpha), 
	\ \  t > 0,\ \lambda > 0.
	\]
	For $\delta=1$, $f^{\alpha,\delta}(t)$  reduces to the interarrival time density of TFPP. Observe that for $\delta=2$,  $f^{\alpha,\delta}(t)$ reduces to the interarrival time density of $\{\hat{N}_{\alpha}(t)\}_{t\ge0}$ which is given in \eqref{renpmf}. The pmf of the renewal process  $\{N^{\alpha, 2}(t)\}_{t \ge 0}$ can be obtained by taking  $\delta=2$ in Eq. (2.8) of Cahoy and Polito (2012b) and it is given by
	\begin{equation}\label{polpmf}
		\Pr\{N^{\alpha, 2}(t)=n\}
		= \lambda^{2n} t^{2\alpha n}
		E_{\alpha,2\alpha n+1}^{2 n}(-\lambda t^\alpha)
		- \lambda^{2(n+1)} t^{2\alpha(n+1)}
		E_{\alpha,2\alpha (n+1)+1}^{2n+2}(-\lambda t^\alpha),
		\ \  n \ge 0.
	\end{equation} 
	Now we show that the pmf \eqref{polpmf} coincides with the pmf \eqref{fracpmf}. From Eq. (3.6) of Beghin and Orsingher (2010), we get
	\begin{align}\label{4}
		\lambda^{2n+1}t^{(2n+1)\alpha}E_{\alpha,(2n+1)\alpha+1}^{2n+2}(-\lambda t^{\alpha}) &=  \lambda^{2n+1}t^{(2n+1)\alpha}E_{\alpha,(2n+1)\alpha+1}^{2n+1}(-\lambda t^{\alpha})\nonumber\\
		&\ \ -\lambda^{2(n+1)}t^{2\alpha(n+1)}E_{\alpha,\alpha(2n+2)+1}^{2n+2}(-\lambda t^{\alpha}).
	\end{align}
	From \eqref{recurmittag}, we have
	\begin{equation}\label{4.1}
		\lambda^{2n}t^{2n\alpha}E_{\alpha,2\alpha n+1}^{2n}(-\lambda t^{\alpha}) =  \lambda^{2n}t^{2n\alpha}E_{\alpha,2n\alpha+1}^{2n+1}(-\lambda t^{\alpha})+\lambda^{2n+1}t^{(2n+1)\alpha}E_{\alpha,\alpha(2n+1)}^{2n+1}(-\lambda t^{\alpha}).
	\end{equation}
	Using \eqref{4} in \eqref{fracpmf}, we get
	\begin{align*}
		\hat{p}_{\alpha}(n,t)&=\lambda^{2n}t^{2n\alpha}E_{\alpha,2n\alpha+1}^{2n+1}(-\lambda t^{\alpha})+\lambda^{2n+1}t^{(2n+1)\alpha}E_{\alpha,(2n+1)\alpha+1}^{2n+1}(-\lambda t^{\alpha})\\
		&\ \ -\lambda^{2(n+1)}t^{2\alpha(n+1)}E_{\alpha,\alpha(2n+2)+1}^{2n+2}(-\lambda t^{\alpha}).
	\end{align*}
	Thus, on substituting \eqref{4.1} in the above equation, we get $\hat{p}_{\alpha}(n,t)=\Pr\{N^{\alpha, 2}(t)=n\}$. Therefore, the processes $\{\hat{N}_{\alpha}(t)\}_{t\ge0}$ and $\{N^{\alpha, 2}(t)\}_{t \ge 0}$ have the same distribution.

	Beghin and Orsingher (2009) obtained the following time-changed relationship  for  the TFPP  $\{N_{\alpha}(t)\}_{ t \ge 0}$ : 
	\begin{equation*}
		{N}_{\alpha}(t)\overset{d}{=} N({T}_{2\alpha}(t)), \ \ \text{$t > 0$}, 
	\end{equation*}
	where the Poisson process  $\{N(t)\}_{ t \ge 0}$  is independent with the random time process $\{T_{2\alpha}(t)\}_{t > 0}$. Note that the distribution  of $\{T_{2\alpha}(t)\}_{t > 0}$ is given by the folded solution of the following fractional diffusion equation (see Orsingher and Beghin (2004)):  
	\begin{equation}\label{diff}
		\frac{\partial^{2 \alpha}}{\partial t^{2 \alpha}}u(x,t)=\frac{\partial^{2}}{\partial x^{2}}u(x,t),\ \ x\in\mathbb{R},\ t > 0,
	\end{equation}
	with $u(x,0)=\delta(x)$ for $0<\alpha \le 1$ and $\dfrac{\partial^{}}{\partial t}u(x,0)=0$ for $1/2<\alpha  \le1$. Here, $\delta(x)$ is the Dirac delta function.
	The solution $u(x,t)=u_{2\alpha}(x,t)$ of \eqref{diff} is given by 
	\begin{equation}\label{u2b}
		u_{2\alpha}(x,t)=\frac{1}{2t^{\alpha}}W_{-\alpha, 1-\alpha}\Big(-\frac{|x|}{t^{\alpha}}\Big), \ \ t>0, \ x\in \mathbb{R},
	\end{equation}
	where $W_{-\alpha, 1-\alpha}(\cdot)$ is the Wright function defined in \eqref{w}.
	Let 
	\begin{equation}\label{ubarxt}
		\bar{u}_{2\alpha}(x,t)=\begin{cases*}
			2u_{2\alpha}(x,t),\ \  x>0, \\
			0,\ x<0,
		\end{cases*}
	\end{equation}
	be the folded solution to \eqref{diff} whose Laplace transform  is given by (see Orsingher and Beghin (2004))
	\begin{equation}\label{lta}
		\int_{0}^{\infty}e^{-st}\bar{u}_{2\alpha}(x,t)\mathrm{d}t=s^{\alpha-1}e^{-x s^{\alpha}}, \ \ x>0.
	\end{equation}	
	Next, we establish a time-changed relationship between $\{\hat{N}_{\alpha}(t)\}_{t\ge0}$ and $\{\hat{N}(t)\}_{t\ge0}$. 
	\begin{theorem}
		The following relationship holds for the process $\{\hat{N}_{\alpha}(t)\}_{t\ge0}$ :
		\begin{equation}\label{nt2alphat}
			\hat{N}_{\alpha}(t)\overset{d}{=}\hat{N}(T_{2\alpha}(t)),\ \ t > 0,
		\end{equation}
		where the process $\{\hat{N}(t)\}_{t\ge0}$ is independent with $\{T_{2\alpha}(t)\}_{t > 0}$.
	\end{theorem}
	\begin{proof}
		The Laplace transform of  pgf  \eqref{pgf}  of $\{\hat{N}_{\alpha}(t)\}_{t\ge0}$ can be obtained in the following form by using \eqref{lapin}:
		\begin{align}\label{sin}
			\tilde{\hat{G}}_{\alpha}(u,s)
			&=\frac{\sqrt{u}+1}{2\sqrt{u}}\dfrac{s^{\alpha-1}}{s^{\alpha}+\lambda(1-\sqrt{u})}+\frac{\sqrt{u}-1}{2\sqrt{u}}\dfrac{s^{\alpha-1}}{s^{\alpha}+\lambda(1+\sqrt{u})}\nonumber\\
			&=\int_{0}^{\infty}\frac{\sqrt{u}+1}{2\sqrt{u}}s^{\alpha-1}e^{-\mu(s^{\alpha}+\lambda(1-\sqrt{u}))}\mathrm{d}\mu+\int_{0}^{\infty}\frac{\sqrt{u}-1}{2\sqrt{u}}s^{\alpha-1}e^{-\mu(s^{\alpha}+\lambda(1+\sqrt{u}))}\mathrm{d}\mu\nonumber\\
			&=\int_{0}^{\infty}s^{\alpha-1}e^{-\mu s^{\alpha}}\Big(\frac{\sqrt{u}+1}{2\sqrt{u}}e^{-\lambda (1-\sqrt{u})\mu}+\frac{\sqrt{u}-1}{2\sqrt{u}}e^{-\lambda (1+\sqrt{u})\mu}\Big)\mathrm{d}\mu \nonumber\\
			&=\int_{0}^{\infty}s^{\alpha-1}e^{-\mu s^{\alpha}}\hat{G}(u,\mu)\mathrm{d}\mu, \ \ \text{(using\ (\ref{pgf2pp}))}\nonumber\\
			&=\int_{0}^{\infty}\hat{G}(u,\mu)\int_{0}^{\infty}e^{-st}\bar{u}_{2\alpha}(\mu,t)\mathrm{d}t \ \mathrm{d}\mu, \ \  \text{(using\ (\ref{lta}))}\nonumber\\
			&=\int_{0}^{\infty}e^{-st}\Big(\int_{0}^{\infty}\hat{G}(u,\mu)\bar{u}_{2\alpha}(\mu,t)\mathrm{d}\mu\Big)\mathrm{d}t \nonumber.
		\end{align}
		By the uniqueness of Laplace transform, we get
		\begin{equation}\label{2.16}
			\hat{G}_{\alpha}(u,t)=\int_{0}^{\infty}\hat{G}(u,\mu)\bar{u}_{2\alpha}(\mu,t)\mathrm{d}\mu.
		\end{equation}
		This completes the proof.
	\end{proof}
	\begin{remark}
		For $\alpha=1/2$, the diffusion equation \eqref{diff} reduces to the heat equation
		\begin{equation*}
			\begin{cases}
				\dfrac{\partial }{\partial t}u(x,t)=\dfrac{\partial^2 }{\partial x^2}u(x,t), 
				\ \ t>0,\ x\in\mathbb{R},\\[6pt]
				u(x,0)=\delta(x),
			\end{cases}
		\end{equation*}
		and the random time $\{T_{1}(t)\}_{t > 0}$ becomes a reflecting Brownian motion $\{|B(t)|\}_{t > 0}$. Thus, from \eqref{nt2alphat} we get $\hat{N}_{1/2}(t)\overset{d}{=}\hat{N}(|B(t)|)$, that is, $\hat{N}_{1/2}(t)$ coincides with $\hat{N}(t)$ at a Brownian time. 
		
		Moreover, for $\alpha = 1/3$, the density of $T_{2/3}(t)$ is given by (see Orsingher and Beghin (2009), Theorem 4.1)
		$$\Pr\{T_{2/3}(t) \in \mathrm{d}x\}/\mathrm{d}x = \frac{3}{2(3t)^{1/3}} \mathcal{A}\left(\frac{|x|}{(3t)^{1/3}}\right),$$
		where
		\begin{equation*}\label{a}
			\mathcal{A}(x) = \frac{1}{\pi} \int_0^\infty \cos\left(wx + \frac{w^3}{3}\right) \mathrm{d}w, 
		\end{equation*}
		is the Airy function. Hence, from \eqref{pmf2pp} and \eqref{nt2alphat} we can write  
		\begin{align*}
			\Pr\{\hat{N}_{1/3}(t) = n\} &= \int_{0}^{\infty} \hat{p}(n,x)\Pr\{T_{2/3}(t) \in \mathrm{d}x\}\\ &= \int_{0}^{\infty}\Big(\frac{(\lambda x)^{2n}}{(2n)!}+\frac{(\lambda x)^{2n+1}}{(2n+1)!}\Big)e^{-\lambda x}\frac{3}{2(3t)^{1/3}} \mathcal{A}\left(\frac{|x|}{(3t)^{1/3}}\right)\mathrm{d}x. 
		\end{align*}
	\end{remark}
	\begin{remark}
		It is known that the density function of an inverse  stable subordinator $\{Y_{\alpha}(t)\}_{t\ge 0}$, $ 0<\alpha <1$ and $\{T_{2\alpha}(t)\}_{t > 0}$ coincide (see Meerschaert {\it et al.} (2011)). Thus, we have
		\begin{equation}\label{sub}
			\hat{N}_{\alpha}(t)\overset{d}{=}\hat{N}(Y_{\alpha}(t)),\ \ \text{$t\ge0$},
		\end{equation}	
		where the inverse stable subordinator  $\{Y_{\alpha}(t)\}_{t\ge 0}$ is independent of $\{\hat{N}(t)\}_{t\ge0}$.
	\end{remark}
	
	It is known that for a L\'evy Process $\{X(t)\}_{t\ge0}$, $\mathbb{E}(X(t))=t\mathbb{E}(X(1))$. From \eqref{mean}, observe that $\mathbb{E}(\hat{N}(t))=\lambda^2 t^{2} E_{1,3}(-2\lambda t)\neq t \mathbb{E}(\hat{N}(1))$, therefore $\{\hat{N}(t)\}_{t\ge0}$ is not a L\'evy Process.

	Next, we give an alternative method to obtain the pgf \eqref{pgf} of  $\{\hat{N}_{\alpha}(t)\}_{t\ge0}$. From \eqref{pgf2pp}, \eqref{u2b}, \eqref{ubarxt}  and \eqref{2.16}, we get
	{\small\begin{align*}
			\hat{G}_{\alpha}(u, t) 
			&= \int_0^{\infty}\Big(\frac{\sqrt{u}+1}{2\sqrt{u}} e^{-\lambda (1-\sqrt{u})y}+ \frac{\sqrt{u}-1}{2\sqrt{u}} e^{-\lambda (1+\sqrt{u})y}\Big) \frac{1}{t^\alpha} W_{-\alpha, 1-\alpha}(-y t^{-\alpha}) \mathrm{d}y \\
			&=\int_0^{\infty}\Big( \frac{\sqrt{u}+1}{2\sqrt{u}} e^{-\lambda (1-\sqrt{u})y}  +\frac{\sqrt{u}-1}{2\sqrt{u}} e^{-\lambda (1+\sqrt{u})y }\Big) \frac{\mathrm{d}y}{t^\alpha} \frac{1}{2\pi i} \int_{Ha}\frac{e^{v-yv^\alpha t^{-\alpha}}}{v^{1-\alpha}}\mathrm{d}v, \ \ (\text{using \eqref{con1}})\\
			&= \frac{t^{-\alpha}}{2\pi i} \int_{Ha} \frac{e^v}{v^{1-\alpha}} \Big(\frac{\sqrt{u}+1}{2\sqrt{u}} \int_0^{\infty} e^{-(\lambda (1-\sqrt{u})+ v^\alpha t^{-\alpha})y} \mathrm{d}y+\frac{\sqrt{u}-1}{2\sqrt{u}}\int_0^{\infty} e^{-(\lambda (1+\sqrt{u})+ v^\alpha t^{-\alpha})y} \mathrm{d}y \Big)\mathrm{d}v\\
			&= \frac{\sqrt{u}+1}{2\sqrt{u}}\frac{t^{-\alpha}}{2\pi i} \int_{Ha} \frac{e^v v^{\alpha-1}}{v^\alpha t^{-\alpha} -(- \lambda (1-\sqrt{u}))} \mathrm{d}v+\frac{\sqrt{u}-1}{2\sqrt{u}}\frac{t^{-\alpha}}{2\pi i} \int_{Ha} \frac{e^v v^{\alpha-1}}{v^\alpha t^{-\alpha} -(- \lambda (1+\sqrt{u}))} \mathrm{d}v \\
			&= \frac{\sqrt{u}+1}{2\sqrt{u}}E_{\alpha, 1}(-\lambda (1-\sqrt{u}) t^\alpha)+\frac{\sqrt{u}-1}{2\sqrt{u}}E_{\alpha, 1}(-\lambda (1+\sqrt{u}) t^\alpha),
	\end{align*}}
	where in last step we have used \eqref{con2}.
	
	Next, we present an alternative representation of the pmf \eqref{fracpmf} in terms of derivatives of the Mittag-Leffler function.
	\begin{proposition}\label{alternate}
		The pmf of the process $\{\hat{N}_{\alpha}(t)\}_{t\ge0}$ is given by
		{\small\begin{equation*}
				\hat{p}_{\alpha}(n,t)=\frac{1}{(2n)!}\frac{\mathrm{d}^{2n}}{\mathrm{d}s^{2n}}\left[s^{2n-1}E_{\alpha,1}\left(-\frac{\lambda t^{\alpha}}{s}\right)\right]_{s=1}+\frac{1}{(2n+1)!}\frac{\mathrm{d}^{2n+1}}{\mathrm{d}s^{2n+1}}\left[s^{2n}E_{\alpha,1}\left(-\frac{\lambda t^{\alpha}}{s}\right)\right]_{s=1}, \ \ \text{$n \ge 0 $}.		
		\end{equation*}}
	\end{proposition}
	\begin{proof}
		From \eqref{ubarxt} and \eqref{nt2alphat}, we have 
		\begin{equation*}
			\hat{p}_{\alpha}(n,t)=\int_{0}^{\infty}\hat{p}(n,y)\bar{u}_{2\alpha}(y,t)\mathrm{d}y.
		\end{equation*}
		Using \eqref{pmf2pp} and \eqref{u2b} in the above equation, we get
		{\small\begin{align*}
				\hat{p}_{\alpha}(n,t)&=\int_{0}^{\infty}\left(\frac{(\lambda y)^{2n}}{(2n)!}+\frac{(\lambda y)^{2n+1}}{(2n+1)!}\right)e^{-\lambda y}t^{-\alpha}W_{-\alpha,1-\alpha}(-yt^{-\alpha})\mathrm{d}y\\
				&=\int_{0}^{\infty}\left(\frac{(\lambda y)^{2n}}{(2n)!}+\frac{(\lambda y)^{2n+1}}{(2n+1)!}\right)e^{-\lambda y}t^{-\alpha}\sum_{m=0}^{\infty}\frac{(-yt^{-\alpha})^{m}}{m!\Gamma(-\alpha m+1-\alpha)}\mathrm{d}y, \ \ (\text{using \eqref{w}})\\
				&=\sum_{m=0}^{\infty}\frac{(-1)^{m}t^{-\alpha-\alpha m}}{m!\Gamma(-\alpha m+1-\alpha)}\Bigg(\frac{\lambda ^{2n}}{(2n)!}\int_{0}^{\infty}e^{-\lambda y}y^{2n+m}\mathrm{d}y
				+\frac{\lambda ^{2n+1}}{(2n+1)!}\int_{0}^{\infty}e^{-\lambda y}y^{2n+m+1}\mathrm{d}y\Bigg)\\
				&=\sum_{m=0}^{\infty}\frac{(-t^{\alpha}\lambda)^{-m}t^{-\alpha}}{\Gamma(-\alpha m+1-\alpha)}\Bigg(\frac{(2n+m)\cdots (m+1)}{\lambda(2n)!}
				+\frac{(2n+m+1)\cdots (m+1)}{\lambda(2n+1)!}\Bigg)\\
				&=\frac{t^{-\alpha}}{\lambda(2n)!}\frac{\mathrm{d}^{2n}}{\mathrm{d}s^{2n}}\left[\sum_{m=0}^{\infty}s^{2n+m}\frac{(-t^{\alpha}\lambda)^{-m}}{\Gamma(-\alpha m+1-\alpha)}\right]_{s=1}\\
				& \ \ +\frac{t^{-\alpha}}{\lambda(2n+1)!}\frac{\mathrm{d}^{2n+1}}{\mathrm{d}s^{2n+1}}\left[\sum_{m=0}^{\infty}s^{2n+m+1}\frac{(-t^{\alpha}\lambda)^{-m}}{\Gamma(-\alpha m+1-\alpha)}\right]_{s=1}\\
				&=\frac{t^{-\alpha}}{\lambda(2n)!}\frac{\mathrm{d}^{2n}}{\mathrm{d}s^{2n}}\left[s^{2n}E_{-\alpha,1-\alpha}\left(-\frac{s}{\lambda t^{\alpha}}\right)\right]_{s=1}+\frac{t^{-\alpha}}{\lambda(2n+1)!}\frac{\mathrm{d}^{2n+1}}{\mathrm{d}s^{2n+1}}\left[s^{2n+1}E_{-\alpha,1-\alpha}\left(-\frac{s}{\lambda t^{\alpha}}\right)\right]_{s=1}.
		\end{align*}}
		Thus, the  proof follows on using the identity given in Eq. (5.1) of Beghin and Orsingher (2009).
	\end{proof}

	Next, we derive the factorial moments of the process $\{\hat{N}_{\alpha}(t)\}_{t\ge0}$ using its pgf.
	\begin{proposition}
		The $r$th factorial moment $\hat{\psi}_\alpha (r,t)=\mathbb{E}\big(\hat{N}_{\alpha}(t)(\hat{N}_{\alpha}(t)-1)\cdots (\hat{N}_{\alpha}(t)-r+1)\big)$ of $\{\hat{N}_{\alpha}(t)\}_{t\ge0}$ is given by
		\begin{equation}\label{moment}
			\hat{\psi}_\alpha (r,t)=r!\lambda^{2r}t^{2\alpha r}E^{r}_{\alpha,2\alpha r+1}(-2\lambda t^{\alpha}), \ \  r \ge 1.
		\end{equation}
	\end{proposition}
	\begin{proof}
		The Laplace transform of the pgf \eqref{pgf} of $\{\hat{N}_{\alpha}(t)\}_{t\ge0}$ is  (see Beghin and Orsingher (2010), Eq. (3.15))
		\begin{equation}\label{hh}
			\tilde{\hat{G}}_\alpha(u,s)=\frac{s^{2\alpha-1}+2\lambda s^{\alpha-1}}{s^{2\alpha}+2\lambda s^{\alpha}+\lambda^{2}(1-u)}.
		\end{equation}
		The $r$th derivative of \eqref{hh}  evaluated at $u=1$ is 
		\begin{equation*}
			\frac{\partial^{r}}{\partial u^{r}}\tilde{\hat{G}}_\alpha(u,s)\Big|_{u=1}=\frac{r!\lambda^{2r}(s^{2\alpha-1}+2\lambda s^{\alpha-1})}{(s^{2\alpha}+2\lambda s^{\alpha}+\lambda^{2}(1-u))^{r+1}}\Big|_{u=1}=\frac{r!\lambda^{2r}s^{-\alpha r-1}}{(s^{\alpha}+2\lambda )^{r}}.
		\end{equation*}
		The result follows on taking the inverse Laplace transform in the above equation by using the formula \eqref{lapin}.  
	\end{proof}
	\begin{remark}\label{comvar}
		From \eqref{moment}, the variance of  $\{\hat{N}_{\alpha}(t)\}_{t\ge0}$ is given by
		\begin{equation*}
			\operatorname{Var}(\hat{N}_{\alpha}(t))=\lambda^{2}t^{2\alpha }E_{\alpha,2\alpha +1}(-2\lambda t^{\alpha})\big(1-\lambda^{2}t^{2\alpha }E_{\alpha,2\alpha +1}(-2\lambda t^{\alpha})\big)+2\lambda^{4}   t^{4\alpha}E^{2}_{\alpha,4\alpha +1}(-2\lambda t^{\alpha}).
		\end{equation*}
		In particular for $\alpha =1$, we get the variance of $\{\hat{N}(t)\}_{t\ge0}$. Suyono (2002) in p. 34  discussed a renewal process whose interarrival time follows Gamma distribution and obtained its mean and covariance.  Note that $\operatorname{Var}(\hat{N}(t))$ coincides with the variance of the process studied in p. 34 of Suyono (2002) which can be obtained by taking $t_1=t_2$ in its covariance expression. 
	\end{remark}
	
	\begin{remark}
		Let $Y_1, Y_2, \dots, Y_n$ be i.i.d. random variables with distribution $F(y) = \Pr\{Y < y\}$ and $\hat{G}(u,t)=\sum_{n=0}^\infty u^n\hat{p}(n,t)$ be the pgf of $\{\hat{N}(t)\}_{t \ge 0}$ as given in \eqref{pgf2pp}. Then 
		\begin{align*}
			&\Pr\Big\{ \max_{1 \le j \le \hat{N}_\alpha(t)} Y_j < y \Big\}\\
			&=\sum_{n=0}^\infty (\Pr\{Y < y\})^n \Pr\{\hat{N}_\alpha(t) = n\} \nonumber\\
			&=\int_0^\infty \hat{G}(F(y), s) \Pr\{T_{2\alpha}(t) \in \mathrm{d}s\}, \  \text{(using \eqref{nt2alphat})} \nonumber\\
			&= \int_0^\infty \Big(\frac{\sqrt{F(y)}+1}{2\sqrt{F(y)}}e^{\big(-\lambda (1-\sqrt{F(y)})s\big)}+\frac{\sqrt{F(y)}-1}{2\sqrt{F(y)}}e^{\big(-\lambda (1+\sqrt{F(y)})s\big)}\Big)\\& \ \ \times t^{-\alpha} W_{-\alpha, 1-\alpha}(-s t^{-\alpha})\mathrm{d}s, \ \text{(using  \eqref{pgf2pp} and \eqref{ubarxt})}\\
			&=\frac{\sqrt{F(y)}+1}{2\sqrt{F(y)}}E_{\alpha,1}{\big(-\lambda \big(1-\sqrt{F(y)}\big)t^\alpha \big)}+\frac{\sqrt{F(y)}-1}{2\sqrt{F(y)}}E_{\alpha,1}{\big(-\lambda \big(1+\sqrt{F(y)}\big)t^\alpha \big)}.
		\end{align*}
		Here, in the last step, we have used Eq. (2.13) of Beghin and Orsingher (2010). Similarly, it can be shown that
		\begin{align*}
			\Pr\Big\{ \min_{1 \le j \le \hat{N}_\alpha(t)} Y_j >y \Big\}&=\frac{\sqrt{\bar{F}(y)}+1}{2\sqrt{\bar{F}(y)}}E_{\alpha,1}{\Big(-\lambda \big(1-\sqrt{\bar{F}(y)}\big)t^\alpha \Big)}\\
			&\ \ +\frac{\sqrt{\bar{F}(y)}-1}{2\sqrt{\bar{F}(y)}}E_{\alpha,1}{\Big(-\lambda \big(1+\sqrt{\bar{F}(y)}\big)t^\alpha \Big)},
		\end{align*}
		where $\bar{F}(y)=1-F(y)$.
	\end{remark}
	Next, we obtain the mgf $\hat{M}_{\alpha}(u,t) = \mathbb{E}(e^{-u\hat{N}_\alpha(t)})$, $u \ge 0$ of the process $\{\hat{N}_{\alpha}(t)\}_{t\ge0}$ .
	\begin{proposition}
		The mgf $\hat{M}_{\alpha}(u,t)$ of  $\{\hat{N}_{\alpha}(t)\}_{t\ge0}$ solves the following fractional differential equation:
		\begin{equation}\label{mgfdiff}
			\frac{\partial^{2\alpha} }{\partial t^{2 \alpha}} \hat{M}_{\alpha}(u,t)
			+ 2\lambda \frac{\partial^{\alpha} }{\partial t^{\alpha}}\hat{M}_{\alpha}(u,t)
			= \lambda^2 (e^{-u}-1)\hat{M}_{\alpha}(u,t), 
			\ \  0 < \alpha \leq 1,
		\end{equation}
		with initial conditions $\hat{M}_\alpha(u,0)=1$ and $\dfrac{\partial^{}}{\partial t}\hat{M}_{\alpha}(u,0)=0$ for $1/2 < \alpha \le 1$,  and its solution is given by 
		\begin{equation}\label{mgf}
			\hat{M}_{\alpha}(u,t) = \frac{\sqrt{e^{-u}}+1}{2\sqrt{e^{-u}}} 
			E_{\alpha,1}\!\big(-\lambda(1-\sqrt{e^{-u}})t^{\alpha}\big)
			+ \frac{\sqrt{e^{-u}}-1}{2\sqrt{e^{-u}}} 
			E_{\alpha,1}\!\big(-\lambda(1+\sqrt{e^{-u}})t^{\alpha}\big).
		\end{equation}
	\end{proposition}
	
	\begin{proof}
		We apply the Laplace transform on both sides of \eqref{mgfdiff} by using \eqref{lc} to obtain
		\begin{equation*}
			(s^{2\alpha} + 2\lambda s^{\alpha}) \tilde{\hat{M}}_{\alpha}(u,s) 
			- (s^{2\alpha-1} + 2\lambda s^{\alpha-1})
			= \lambda^2 (e^{-u}-1)\tilde{\hat{M}}_{\alpha}(u,s).
		\end{equation*}
		Thus,
		\begin{equation}\label{44}
			\tilde{\hat{M}}_{\alpha}(u,s) = 
			\frac{s^{2\alpha-1} + 2\lambda s^{\alpha-1}}{s^{2\alpha} + 2\lambda s^{\alpha} + \lambda^2(1-e^{-u})}.
		\end{equation}
		We observe that the governing fractional equation \eqref{mgfdiff} and  the Laplace transform  \eqref{44} coincide with Eq. (2.3a) and Eq. (2.6) of Orsingher and Beghin (2004), respectively, with $c^2\beta^2 = \lambda^2(1-e^{-u})$. Hence, the inverse Laplace transform of \eqref{44} can be obtained by using Theorem 2.1 of Orsingher and Beghin (2004) and it leads to the expression given in \eqref{mgf}. 
	\end{proof}
	\begin{remark}
		On taking $r$th order derivative with respect to $u$ on  both sides of \eqref{44}, we get
		\begin{align*}
			\frac{\partial^{r}}{\partial u^{r}}\tilde{\hat{M}}_{\alpha}(u,s) \Big|_{u=0} &=\frac{\sum_{k=0}^r (-1)^r S(r,k) k!\lambda^{2k}e^{-ku}(s^{2\alpha-1}+2\lambda s^{\alpha-1})}{(s^{2\alpha}+2\lambda s^{\alpha}+\lambda^{2}(1-e^{-u}))^{k+1}} \Big|_{u=0}\\&=\frac{\sum_{k=0}^r (-1)^r S(r,k) k!\lambda^{2k}s^{-\alpha k-1}}{(s^{\alpha}+2\lambda )^{k}},
		\end{align*}
		where $S(r,k)=\displaystyle\dfrac{1}{k!}\sum_{j=0}^k(-1)^{k-j}\binom{k}{j}j^r$ is the Stirling number of the second kind. Thus, on applying the inverse Laplace transform in the above equation and by using  \eqref{lapin}, we obtain 
		\begin{equation*}\label{f}
			\frac{\partial^{r}}{\partial u^{r}}\hat{M}_{\alpha}(u,t) \Big|_{u=0}=\sum_{k=0}^r (-1)^r S(r,k) k!\lambda^{2k}t^{2\alpha k}E^{k}_{\alpha,2\alpha k+1}(-2\lambda t^{\alpha}).
		\end{equation*}
		Therefore, the $r$th order moment of $\{\hat{N}_{\alpha}(t)\}_{t\ge0}$ is given by
		\begin{align}\label{mom}
			\mathbb{E}\big((\hat{N}_{\alpha}(t))^r\big)&=(-1)^r\frac{\partial^{r}}{\partial u^{r}}\hat{M}_{\alpha}(u,t) \Big|_{u=0}\nonumber\\
			&=\sum_{k=0}^r S(r,k) k!\lambda^{2k}t^{2\alpha k}E^{k}_{\alpha,2\alpha k+1}(-2\lambda t^{\alpha})
			\\&=\sum_{k=0}^r S(r,k)\hat{\psi}_\alpha (k,t)\label{order},
		\end{align}
		where $\hat{\psi}_\alpha (.,t)$ is the  factorial moment of $\{\hat{N}_{\alpha}(t)\}_{t\ge0}$ given in \eqref{moment}.
	\end{remark}
	Next, we evaluate the moments and covariance (in the Laplace domain) of $\{\hat{N}_{\alpha}(t)\}_{t \ge 0}$ by applying the results which hold for any renewal process (see Suyono and Hadi (2018), Eqs. (7), (8) and (18)).
	
	Recall that $\{\hat{N}_{\alpha}(t)\}_{t \ge 0}$ is a  renewal process  with interarrival density $f_{\hat{T}_1^{\alpha}}(\cdot)$ which is given in \eqref{renlap}. For $s > 0$, we have
	\begin{align*}
		\int_{0}^{\infty} e^{-st}\mathbb{E}\big(\hat{N}_{\alpha}(t)\big)\mathrm{d}t
		&=
		\frac{\tilde{f}_{\hat{T}_1^{\alpha}}(s)}{s\big(1-\tilde{f}_{\hat{T}_1^{\alpha}}(s)\big)}\\
		&=\frac{\lambda^2 s^{-\alpha-1}}{(s^{\alpha}+2\lambda)}, \ \text{(using\eqref{renlap})}.
	\end{align*}
	On taking the inverse Laplace transform in the above equation and by applying \eqref{lapin}, we get the mean of $\{\hat{N}_{\alpha}(t)\}_{t \ge 0}$ which is given in \eqref{mean}. Thus, it serves as an alternate method to obtain the mean. Also
	\begin{align*}
		\int_{0}^{\infty} e^{-st}\mathbb{E}\big(\hat{N}_{\alpha}^2(t)\big)\mathrm{d}t
		&=
		\frac{\tilde{f}_{\hat{T}_1^{\alpha}}(s)}{s\Big(1-\tilde{f}_{\hat{T}_1^{\alpha}}(s)\Big)}
		+
		\frac{2(\tilde{f}_{\hat{T}_1^{\alpha}}(s))^2}{s\left(1-\tilde{f}_{\hat{T}_1^{\alpha}}(s)\right)^2}\\
		&=\frac{\lambda^2 s^{-\alpha-1}}{(s^{\alpha}+2\lambda)}+\frac{2\lambda^4 s^{-2\alpha-1}}{(s^\alpha+2\lambda)^2}.
	\end{align*}
	Thus, again by applying the inverse  Laplace transform we can obtain the second order moments of $\{\hat{N}_{\alpha}(t)\}_{t \ge 0}$ which is given in \eqref{order}.
	
	Next for $s_1>0$, $s_2 > 0$, we have
	\begin{align*}
		\int_{0}^{\infty}\!\!\int_{0}^{\infty}&
		e^{-s_1 t_1 - s_2 t_2}
		\mathbb{E}\big(\hat{N}_{\alpha}(t_1)\hat{N}_{\alpha}(t_2)\big) \mathrm{d}t_1 \ \mathrm{d}t_2\\
		&=
		\frac{\big(1-\tilde{f}_{\hat{T}_1^{\alpha}}(s_1)\tilde{f}_{\hat{T}_1^{\alpha}}(s_2)\big)\tilde{f}_{\hat{T}_1^{\alpha}}(s_1+s_2)}
		{s_1 s_2
			\big(1-\tilde{f}_{\hat{T}_1^{\alpha}}(s_1)\big)
			\big(1-\tilde{f}_{\hat{T}_1^{\alpha}}(s_2)\big)
			\big(1-\tilde{f}_{\hat{T}_1^{\alpha}}(s_1+s_2)\big)}\\
		&=\frac{\lambda^2\big((s_1^{\alpha}+\lambda)^2(s_2^{\alpha}+\lambda)^2-\lambda^4\big)}{s_1^{\alpha+1}s_2^{\alpha+1}(s_1+s_2)^{\alpha}\big(s_1^{\alpha}+2\lambda \big)\big(s_2^{\alpha}+2\lambda \big)\big((s_1+s_2)^{\alpha}+2\lambda \big)}.
	\end{align*}
	Thus, the  covariance of $\{\hat{N}_{\alpha}(t)\}_{t \ge 0}$ in the Laplace domain is
	\begin{align*}
		&\int_{0}^{\infty}\int_{0}^{\infty}
		e^{-s_1 t_1 - s_2 t_2}
		\operatorname{Cov}\big(\hat{N}_{\alpha}(t_1),\hat{N}_{\alpha}(t_2)\big) \mathrm{d}t_1 \ \mathrm{d}t_2\\
		&=
		\frac{\lambda^2\big((s_1^{\alpha}+\lambda)^2(s_2^{\alpha}+\lambda)^2-\lambda^4\big)-\lambda^4\big((s_1+s_2)^{2\alpha}+2\lambda(s_1+s_2)^{\alpha}\big)}{s_1^{\alpha+1}s_2^{\alpha+1}(s_1+s_2)^{\alpha}\big(s_1^{\alpha}+2\lambda \big)\big(s_2^{\alpha}+2\lambda \big)\big((s_1+s_2)^{\alpha}+2\lambda \big)}.
	\end{align*}

	Di Crescenzo {\it et al.} (2016) studied an asymptotic behavior of the TFPP and a generalization of it as their parameters become large. Now, we study an asymptotic behavior of $\{\hat{N}_{\alpha}(t)\}_{t\ge0}$ for the large value of $\lambda$.
	\begin{proposition}\label{Q}
		For a fixed $t > 0$, we have 
		\begin{equation*}
			\frac{\hat{N}_{\alpha}(t)}{\mathbb{E}\big( \hat{N}_{\alpha}(t) \big)} 
			\xrightarrow[\lambda \to \infty]{\ \ \mathrm{\Pr}\ \ } 1, \ \ \text{$0<\alpha\le 1$}.
		\end{equation*}
		
	\end{proposition}
	
	\begin{proof}
		The following relation holds due to the triangular inequality:
		\begin{equation*}
			\mathbb{E}\Big( \Big| \frac{\hat{N}_{\alpha}(t)}{\mathbb{E}\big(\hat{N}_{\alpha}(t)\big)} - 1 \Big| \Big) \leq 2.
		\end{equation*}
		Hence, by applying  the dominated convergence theorem, and  using \eqref{fracpmf} and \eqref{mean}, we obtain
		\begin{align*}
			\lim_{\lambda \to \infty} \mathbb{E}\Big( \Big| \frac{\hat{N}_{\alpha}(t)}{\mathbb{E}\big(\hat{N}_{\alpha}(t)\big)} - 1 \Big| \Big) 
			=& \lim_{\lambda \to \infty} \sum_{n=0}^{\infty} \Big|\frac{n}{\lambda^2 t^{2\alpha}E_{\alpha,2\alpha+1}(-\lambda t^{\alpha})} - 1 \Big|
			\Big(\lambda^{2n}t^{2n\alpha}E_{\alpha,2n\alpha+1}^{2n+1}(-\lambda t^{\alpha})\\ &  \ \ +\hspace{1mm}\lambda^{2n+1}t^{(2n+1)\alpha}E_{\alpha,(2n+1)\alpha+1}^{2n+2}(-\lambda t^{\alpha})\Big)\\& =0.
		\end{align*}
		Here, in the last step we have used the following asymptotic result of the three parameter Mittag-Leffler function for large $z$ (see Saxena {\it et al.}\ (2004), Eq. (9)):
		\begin{equation*}
			E_{\alpha,\beta}^{\gamma}(z) \sim O( |z|^{-\gamma} ),
			\ \  |z|>1.
		\end{equation*}
		Thus, the random variable $\dfrac{\hat{N}_{\alpha}(t)}{\mathbb{E}\big(\hat{N}_{\alpha}(t)\big)}$ converges in mean to $1$. Since convergence in
		mean implies convergence in probability, the result follows.
	\end{proof}
	
	Using the $r$th order moment of  $\{\hat{N}_{\alpha}(t)\}_{t\ge0}$ given in \eqref{mom}, the proof of the following result follows along the same lines as that of  Proposition \eqref{Q}. Thus, it is omitted. 
	\begin{proposition}\label{L}
		For a fixed $t > 0$, we have 
		\begin{equation*}
			\frac{\big(\hat{N}_{\alpha}(t)\big)^r}{\mathbb{E}\big( \hat{N}_{\alpha}(t) \big)^r} 
			\xrightarrow[\lambda \to \infty]{\ \ \mathrm{\Pr}\ \ } 1, \ \ r \ge 1, \ \text{$0<\alpha\le 1$}.
		\end{equation*}
	\end{proposition}
	
	\begin{remark}
		It is known that a family of random variables $U^{(\lambda)}$ exhibits cut-off behavior at mean times if the following holds (see Barrera  {\it et al.} (2009), Definition 1):
		\[
		\frac{U^{(\lambda)}}{\mathbb{E}(U^{(\lambda)})} \xrightarrow[\lambda \to \infty]{\mathrm{Prob}} 1.
		\]
		Thus, from the Propositions \eqref{Q} and \eqref{L}, it follows that the processes  $\hat{N}_{\alpha}(t)$ and $\big(\hat{N}_{\alpha}(t)\big)^r$, $ \ r \ge 1$ exhibit  cut-off behavior at mean times with respect to the associated parameter. 
	\end{remark}

	\begin{proposition}\label{prop3.6}
		The one-dimensional distributions of  $\{\hat{N}_{\alpha}(t)\}_{t\ge0}$ are not infinitely divisible.
	\end{proposition}
	\begin{proof}

		From \eqref{sub} and \eqref{selfsimi}, we have
		\begin{align}\label{22}
			\lim_{t \to \infty} \frac{\hat{N}_\alpha(t)}{t^\alpha}&\overset{d}{=}\lim_{t \to \infty} \frac{\hat{N}(Y_\alpha(t))}{t^\alpha}\overset{d}{=} \lim_{t \to \infty} \frac{\hat{N} \big( t^\alpha Y_\alpha(1) \big)}{t^\alpha Y_\alpha(1)} Y_\alpha(1).
		\end{align}
		Note that $\{\hat{N}(t)\}_{t\ge 0}$ is a renewal process with Gamma($\lambda, 2$) distributed interarrival times (see Beghin and Orsingher (2010), Section (3.4)). Thus, from the law of large numbers, it follows that 
		\begin{equation}\label{limitS}
			\lim\limits_{t\to\infty}\frac{\hat{N}(t)}{t}=\frac{\lambda}{2},\ \ \text{with probability 1}.
		\end{equation}
		Using \eqref{limitS} in \eqref{22}, we obtain
		\begin{align*}
			\lim_{t \to \infty} \frac{\hat{N}_\alpha(t)}{t^\alpha}\overset{d}{=} \frac{\lambda}{2}Y_\alpha(1).
		\end{align*}
		Let us assume that  $\hat{N}_{\alpha}(t)$ is infinitely divisible. Then, $\dfrac{\hat{N}_\alpha(t)}{t^\alpha}$ is also infinitely divisible (see Steutel and Van Harn (2004), Proposition 2.1). As the limit of a sequence of  infinitely divisible random variables  is infinitely divisible (see Steutel and Van Harn (2004), Proposition 2.2), it follows that $Y_\alpha(1)$ is infinitely divisible. It leads to a contradiction as  $ Y_\alpha(1) $ is not infinitely divisible (see Kumar and Nane (2018)).  
	\end{proof}
	
	Next, we show that the process  $\{\hat{N}(t)\}_{ t \ge 0}$ can be obtained as a time-changed version of $\{\hat{N}_{\alpha}(t)\}_{t\ge0}$, where the time-change is done by a random time  process $\{H^\alpha(t)\}_{t>0}, 0 < \alpha \le 1$, whose pdf $f_{H^\alpha(t)}(s, t)$,  $s>0$ is defined as (see Cahoy and Polito (2012a))
	$$
	f_{H^\alpha(t)}(s, t)= t^{-\frac{1}{\alpha}} H_{1, 1}^{1, 0} \left[ t^{-\frac{1}{\alpha}} s \left| \begin{matrix} (1-1/\alpha, 1/\alpha) \\ (0, 1) \end{matrix} \right. \right].
	$$
	Here, $H_{1, 1}^{1, 0}(\cdot)$ denotes the $H$-function (see Mathai {\it et al.} (2010)). 
	The Mellin transform of this density function is given by
	\begin{equation}\label{88}
		\int_0^\infty s^{\nu-1} f_{H^\alpha(t)}(s, t)\mathrm{d}s = \frac{\Gamma(\nu) t^{\frac{\nu-1}{\alpha}}}{\Gamma(1-1/\alpha + \nu/\alpha)}, \ \ t > 0, \  \nu \in \mathbb{R}.
	\end{equation}
	
	\begin{theorem}
		Let the random time process  $\{H^\alpha(t)\}_{t >  0}, 0 < \alpha \le 1$ be  independent of $\{\hat{N}_\alpha(t)\}_{t \ge 0}$. Then the  following relation holds:
		$$
		\hat{N}(t) \stackrel{d}{=} \hat{N}_\alpha(H^\alpha(t)), \ \  t > 0.
		$$
	\end{theorem}
	\begin{proof}
		Using \eqref{fracpmf}, we have
		{\small\begin{align*}
				\int_0^\infty& \hat{G}_{\alpha}(u, s) f_{H^\alpha(t)}(s, t) \mathrm{d}s\\
				&= \int_0^\infty \sum_{n=0}^\infty u^n \hat{p}_\alpha(n, s) f_{H^\alpha(t)}(s, t)\mathrm{d}s\\
				&= \int_0^\infty \sum_{n=0}^\infty u^n \Big(\lambda^{2n}s^{2n\alpha}E^{2n+1}_{\alpha,2n\alpha+1}(-\lambda s^{\alpha})+\lambda^{2n+1}s^{(2n+1)\alpha}E^{2n+2}_{\alpha,(2n+1)\alpha+1}(-\lambda s^{\alpha})\Big)f_{H^\alpha(t)}(s, t)\mathrm{d}s\\
				&=\int_0^\infty \sum_{n=0}^\infty u^n \Bigg(\frac{\lambda^{2n}s^{2n\alpha}}{2\pi i \Gamma(2n+1)}\int_{c-i\infty}^{c+i\infty} \frac{\Gamma(z)\Gamma(2n+1-z)}{\Gamma((2n-z)\alpha+1)} \\ & \ \  +\frac{\lambda^{2n+1}s^{(2n+1)\alpha}}{2\pi i \Gamma(2n+2)}\int_{c-i\infty}^{c+i\infty} \frac{\Gamma(z)\Gamma(2n+2-z)}{\Gamma((2n-z+1)\alpha+1)}  \Bigg)(\lambda s^{\alpha})^{-z}\mathrm{d}zf_{H^\alpha(t)}(s, t)\mathrm{d}s, \  (\text{using \eqref{cont}})\\
				&=\sum_{n=0}^\infty \frac{u^n}{2\pi i}  \Bigg(\frac{\lambda^{2n}}{\Gamma(2n+1)}\int_{c-i\infty}^{c+i\infty} \frac{\Gamma(z)\Gamma(2n+1-z)}{\Gamma((2n-z)\alpha+1)}\lambda^{-z}\int_0^\infty s^{2n\alpha-\alpha z} f_{H^\alpha(t)}(s, t)\mathrm{d}s \ \mathrm{d}z\\ & \ \  +\frac{\lambda^{2n+1}}{ \Gamma(2n+2)}\int_{c-i\infty}^{c+i\infty} \frac{\Gamma(z)\Gamma(2n+2-z)}{\Gamma((2n-z+1)\alpha+1)}\lambda^{-z}\int_0^\infty s^{(2n+1)\alpha-\alpha z} f_{H^\alpha(t)}(s, t)\mathrm{d}s \ \mathrm{d}z\Bigg)\\
				&=\sum_{n=0}^\infty \frac{u^n \lambda^{-z}}{2\pi i}  \Bigg(\frac{\lambda^{2n}}{\Gamma(2n+1)}\int_{c-i\infty}^{c+i\infty} \frac{\Gamma(z)\Gamma(2n+1-z)}{\Gamma((2n-z)\alpha+1)}\frac{\Gamma((2n-z)\alpha+1)}{\Gamma(2n+1-z)}t^{2n-z}\mathrm{d}z\\ & \ \  +\frac{\lambda^{2n+1}}{ \Gamma(2n+2)}\int_{c-i\infty}^{c+i\infty} \frac{\Gamma(z)\Gamma(2n+2-z)}{\Gamma((2n-z+1)\alpha+1)}\frac{\Gamma((2n-z+1)\alpha+1)}{\Gamma(2n+2-z)}t^{2n+1-z}\mathrm{d}z\Bigg), \  \text{(using \eqref{88})}\\
				&=\sum_{n=0}^\infty u^n \Big(\frac{(\lambda t)^{2n}}{(2n)!}+\frac{(\lambda t)^{2n+1}}{(2n+1)!}\Big)\frac{1}{2\pi i}\int_{c-i\infty}^{c+i\infty} \Gamma(z) (\lambda t)^{-z}\mathrm{d}z\\
				&=\sum_{n=0}^\infty u^n \Big(\frac{(\lambda t)^{2n}}{(2n)!}+\frac{(\lambda t)^{2n+1}}{(2n+1)!}\Big)e^{-\lambda t}, \  \text{(using \eqref{mbe})}\\&=\sum_{n=0}^\infty u^n\hat{p}(n,t)=\hat{G}(u,t).
		\end{align*}}
		Thus, the proof follows.
	\end{proof}
	Orsingher and Polito (2013) studied the bivariate  distribution for the TFPP. Here, we present the  bivariate distribution of the process $\{\hat{N}_{\alpha}(t)\}_{t\ge0}$. Recall that $\{\hat{N}_{\alpha}(t)\}_{t\ge0}$ is renewal process whose density of the interarrival times is given in \eqref{renpmf}.  Note that $\hat{{\tau}}^{{\alpha},l}_h=\hat{T}^{\alpha}_{l+h}-\hat{T}^{\alpha}_l\overset{d}{=}\hat{T} ^{\alpha}_h$, where $\hat{{\tau}}^{{\alpha},l}_h$ represents the length of the time interval separating the $l$-th and $(l+h)$-th events, and  ${\hat{T}^{\alpha}_l}=\inf\{t:{\hat{N}_{\alpha}(t)}=l\} $ denote the random time of occurrence of the $l$-th event for $\{\hat{N}_{\alpha}(t)\}_{t\ge0}$. The distribution of ${\hat{T}^{\alpha}_l}$ is given by (see Beghin and Orsingher (2010), Eq. (3.21))
	\begin{equation}\label{2.4}
		\mathrm{Pr}\{\hat{T}_l^{\alpha}\in\mathrm{d}s\}/\mathrm{d}s = \lambda^{2l}s^{2\alpha l-1}E^{2l}_{\alpha,2\alpha l }(-\lambda s^{\alpha}),\ \ s\ge 0, \ 0<\alpha\le 1.
	\end{equation}
	
	\begin{theorem}\label{thm3.3}
		The bivariate distribution of $\{\hat{N}_{\alpha}(t)\}_{t\ge0}$, $0<\alpha\le 1$ is given by
		\begin{align*}
			\mathrm{Pr}&\{\hat{N}_{\alpha}(s)=l,\hat{N}_{\alpha}(t)=m\}\nonumber\\
			&=\lambda^{2m}\int_0^s u^{2\alpha l - 1} E^{2l}_{\alpha,2\alpha l}(-\lambda u^\alpha) \int_{s-u}^{t-u} v^{2\alpha - 1} E^2_{\alpha,2\alpha}(-\lambda v^\alpha) \times  \\
			& \ \ \Big((t-u-v)^{\alpha (2m-2l-2)} E^{2(m-l)-1}_{\alpha,1+2\alpha (m-l-1)}(-\lambda(t-u-v)^{\alpha})\\
			&\ \ +\lambda(t-v-u-\xi)^{\alpha}
			(t-u-v)^{\alpha(2m-2l-1)}E^{2(m-l)}_{\alpha,\alpha+1 +2\alpha (m-l-1)}(-\lambda(t-u-v)^{\alpha})\Big) \mathrm{d}v \ \mathrm{d}u.
		\end{align*}
	\end{theorem}

	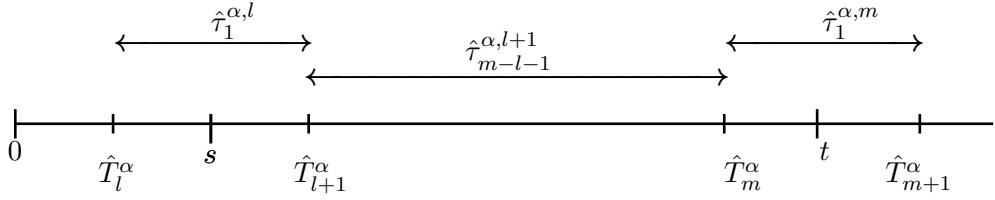
\begin{figure}[htbp]
		\renewcommand{\figurename}{Fig.}
		\centering
		
		\unitlength=0.65mm
		\begin{picture}(50,30)
			\linethickness{1pt}
			\put(-70,10){\line(1,0){200}}
			
			
			\put(-70,13){\line(0,-1){6}}
			\put(-71.5,2.5){\small{$0$}}
			
			\put(-30,12){\line(0,-1){6}}
			\put(-31.5,2){\small{$s$}}
			
			\put(94,13){\line(0,-1){6}}
			\put(94.5,2.5){\small{$t$}}
			
			\put(-30,12){\line(0,-1){6}}
			\put(-31.5,2){\small{$s$}}
			
			\put(-50,25){$\boldsymbol{\xleftrightarrow{\hspace{2.4cm}}}$}
			\put(-30,29){\small{$\hat{\tau}_1^{\alpha,l}$}}

			\put(-11,18){$\boldsymbol{\xleftrightarrow{\hspace{5.4cm}}}$}
			\put(22,23){\small{$\hat{\tau}_{m-l-1}^{\alpha,l+1}$}}

			
			\put(75,25){$\boldsymbol{\xleftrightarrow{\hspace{2.4cm}}}$}
			\put(95,29){\small{$\hat{\tau}_1^{\alpha,m}$}}
			
			\put(-50,12){\line(0,-1){4}}
			\put(-53,-02){\small{$\hat{T}_{l}^\alpha$}}
			
			
			\put(-10,12){\line(0,-1){4}}
			\put(-13,-02){\small{$\hat{T}_{l+1}^\alpha$}}
			
			\put(75,12){\line(0,-1){4}}
			\put(75,-2){\small{$\hat{T}_{m}^{\alpha}$}}
			
			
			\put(115,12){\line(0,-1){4}}
			\put(108,-2){\small{$\hat{T}_{m+1}^{\alpha}$}}

		\end{picture}
		\caption{\tiny The instants of occurrence of the events and the related waiting times.}\label{fig}
	\end{figure}
	\begin{proof}
		Let $ 0 < s \le t $ and $0  < l \le m $, then from the Fig. \ref{fig}, we have
		\begin{align*}
			\mathrm{Pr}&\{\hat{N}_\alpha(s)=l, \hat{N}_\alpha(t)=m\} \nonumber\\
			&= \mathrm{Pr}\{0 < \hat{T}^{\alpha}_{l} \leq s, \, \hat{{\tau}}_1^{\alpha,l} > s - \hat{T}^{\alpha}_{l}, \, \hat{{\tau}}_1^{\alpha,l} < t - \hat{T}^{\alpha}_{l}, \, 0 < \hat{{\tau}}^{\alpha,l+1}_{ m-l-1} < t - \hat{T}^{\alpha}_{l+1}, \, \hat{\tau}_1^{\alpha,m} > t - \hat{T}^{\alpha}_m\}\\
			&=\mathrm{Pr}\big\{0 < \hat{T}^{\alpha}_{l} \leq s, \,  s - \hat{T}^{\alpha}_{l} < \hat{{\tau}}_1^{\alpha,l} < t - \hat{T}^{\alpha}_{l}, \, 0 < \hat{{\tau}}^{\alpha,l+1}_{ m-l-1} < t - \hat{T}^{\alpha}_{l}-\hat{{\tau}}_1^{\alpha,l}, \, \\
			&\hspace{1cm} \hat{\tau}_1^{\alpha,m} > t - \hat{T}^{\alpha}_l-\hat{{\tau}}_1^{\alpha,l}-\hat{{\tau}}^{\alpha,l+1}_{ m-l-1}\big\}.
		\end{align*}
		Since the interarrival times are i.i.d., we have that
		{\small\begin{align*}
				\mathrm{Pr}&\{\hat{N}_\alpha(s)=l, \hat{N}_\alpha(t)=m\}\\
				&= \int_0^s \Pr\{\hat{T}^{\alpha}_{l} \in \mathrm{d}u\} \int_{s-u}^{t-u} \Pr\{ \hat{{\tau}}_1^{\alpha,l} \in  \mathrm{d}v\} \int_0^{t-(u+v)} \Pr\{\hat{{\tau}}^{\alpha,l+1}_{ m-l-1}\in  \mathrm{d}\xi\}\int_{t-(u+v+\xi)}^\infty \Pr\{ \hat{\tau}_1^{\alpha,m} \in  \mathrm{d}\eta\}\\
				&= \int_0^s \Pr\{\hat{T}^{\alpha}_{l} \in \mathrm{d}u\} \int_{s-u}^{t-u} \Pr\{ \hat{{T}}_1^{\alpha} \in  \mathrm{d}v\} \int_0^{t-(u+v)} \Pr\{\hat{{T}}^{\alpha}_{ m-l-1}\in  \mathrm{d}\xi\} \ \  \mathrm{Pr}\{\hat{T}_1^{\alpha} > {(t-u-v-\xi)}\}\\
				&= \int_0^s \lambda^{2l} u^{2\alpha l - 1} E^{2l}_{\alpha,2\alpha l}(-\lambda u^\alpha)  \mathrm{d}u \int_{s-u}^{t-u} \lambda^2 v^{2\alpha - 1} E^2_{\alpha,2\alpha}(-\lambda v^\alpha)  \mathrm{d}v\\
				&\ \ \times \int_0^{t-(u+v)} \lambda^{2 (m-l-1)} \xi^{2\alpha (m-l-1) - 1} E^{2(m-l-1)}_{\alpha,2\alpha (m-l-1)}(-\lambda \xi^\alpha)  \mathrm{d}\xi \Big( E_{\alpha,1}(-\lambda(t-v-u-\xi)^\alpha)\\
				& \ \ + \lambda (t-v-u-\xi)^\alpha E^2_{\alpha,\alpha+1}(-\lambda(t-v-u-\xi)^\alpha) \Big) ,  \ \text{(using \eqref{renpmf}, \eqref{sur} and \eqref{2.4})}  \\
				&\ \ = \int_0^s u^{2\alpha l - 1} E^{2l}_{\alpha,2\alpha l}(-\lambda u^\alpha) \int_{s-u}^{t-u} v^{2\alpha - 1} E^2_{\alpha,2\alpha}(-\lambda v^\alpha) \bigg(\lambda^{2m}\int_0^{t-(u+v)} \xi^{2\alpha (m-l-1) - 1}\\
				& \ \ \times E^{2(m-l-1)}_{\alpha,2\alpha (m-l-1)}(-\lambda \xi^\alpha)  E_{\alpha,1}(-\lambda(t-v-u-\xi)^{\alpha})+\lambda^{2m+1}(t-v-u-\xi)^{\alpha}\\
				&\ \ \times \int_0^{t-(u+v)} \xi^{2\alpha (m-l-1) - 1}E^{2(m-l-1)}_{\alpha,2\alpha (m-l-1)}(-\lambda \xi^\alpha)  E^2_{\alpha,\alpha+1}(-\lambda(t-v-u-\xi)^{\alpha})\bigg) \mathrm{d}\xi \ \mathrm{d}v
				\ \mathrm{d}u.
		\end{align*}}
		Using the following result (see  Haubold {\it et al.} (2011), Eq. (11.7)) in the above equation, we get our desired result:
		\begin{equation*}
			\int_0^x (x-t)^{\beta-1} E_{\alpha, \beta}^{\gamma}(a(x-t)^\alpha) t^{\zeta-1} E_{\alpha, \zeta}^\sigma(at^\alpha) \mathrm{d}t = x^{\beta+\zeta-1} E_{\alpha, \beta+ \zeta}^{\gamma+\sigma}(ax^\alpha),\ \  a\in \mathbb{R},
		\end{equation*}
where $\alpha>0$, $\beta>0$, $\gamma>0$, $\zeta>0$, $\sigma> 0$.
	\end{proof}

	Cahoy {\it et al.} (2010) and Khandakar {\it et al.} (2025) discussed  the nontrivial scaling limits of the marginal distributions of the TFPP and a generalization of it, respectively. Next, we discuss the scaling limits of the marginal distribution of the process $\{{\hat{N}_\alpha(t)}\}_{t\ge 0}$.
	
	Here, we consider the following standardized random variable:
	\begin{equation*}
		\hat{X}=\frac{\hat{N}_\alpha(t)}{\eta}, \ \  
		\text{where}  \ \eta=\mathbb{E}(\hat{N}_\alpha(t))=\lambda ^2t^{2\alpha}E_{\alpha,2\alpha+1}(-2\lambda t^{\alpha}).
	\end{equation*}
	On using  \eqref{asymittg} in $\eta $ for large $t$, we get
	\begin{equation}\label{asymmean}
		\eta \sim \frac{\lambda t^{\alpha}}{2\Gamma(\alpha+1)} .
	\end{equation}
	On substituting $u=e^{-s/\eta}$ in \eqref{pgf}, we get the Laplace transform of $\hat{X}$ 
	\begin{equation}\label{pf0}
		\mathbb{E}(e^{-s\hat{X}})
		=\frac{\sqrt{e^{-s/\eta}}+1}{2\sqrt{e^{-s/\eta}}}E_{\alpha,1}(-\lambda (1-\sqrt{e^{-s/\eta}})t^{\alpha})+\frac{\sqrt{e^{-s/\eta}}-1}{2\sqrt{e^{-s/\eta}}}E_{\alpha,1}(-\lambda (1+\sqrt{e^{-s/\eta}})t^{\alpha}).
	\end{equation}
	Now, we find the asymptotic behavior of \eqref{pf0} for large $t$. Observe that
	\begin{align*}
		\frac{\sqrt{e^{-s/\eta}}+1}{2\sqrt{e^{-s/\eta}}}&=\frac{1}{2}+\frac{1}{2}e^{s/2\eta}\nonumber\\ &= \frac{1}{2}+\frac{1}{2}\big( 1+\frac{s}{2\eta} + \frac{s^2}{8\eta^2} + \cdots\big)
		\nonumber\\ &\sim \frac{1}{2}+\frac{1}{2}\Big( 1+\frac{s}{2\frac{\lambda t^{\alpha}}{2\Gamma(\alpha+1)}} + \frac{s^2}{8(\frac{\lambda t^{\alpha}}{2\Gamma(\alpha+1)})^2} + \cdots\Big), \ \text{(using \eqref{asymmean})}\nonumber\\
		&\sim 1,\  \ \text{as $t\rightarrow\infty$}. 
	\end{align*}
	Similarly,
	\begin{equation*}
		\frac{\sqrt{e^{-s/\eta}}-1}{2\sqrt{e^{-s/\eta}}}= o(1),\ \ \text{as $t\rightarrow\infty$},
	\end{equation*}
	which implies that the term $\displaystyle{\frac{\sqrt{e^{-s/\eta}}-1}{2\sqrt{e^{-s/\eta}}}}  $ vanishes as $t \rightarrow \infty$. Here, we have used some results from asymptotic analysis, for more details on this we refer the reader to Olver (1974).
	Thus, it follows from \eqref{pf0} that
	\begin{align}\label{pf1}
		\mathbb{E}(e^{-s\hat{X}})&\sim E_{\alpha,1}(-\lambda (1-\sqrt{e^{-s/\eta}})t^{\alpha}) \nonumber\\
		&= \sum_{j=0}^ {\infty}\frac{\big(-\lambda t^{\alpha}(1-e^{-s/2\eta})\big)^j}{\Gamma(j\alpha+1)}.
	\end{align}
	As $t\rightarrow\infty$, observe from \eqref{asymmean} that
	\begin{align}\label{pf4}
		\lambda t^{\alpha}(1-e^{-s/2\eta})&\sim\lambda t^{\alpha}\big(\frac{s}{2\frac{\lambda t^{\alpha}}{2\Gamma(\alpha+1)}} -\frac{s^2}{8(\frac{\lambda t^{\alpha}}{2\Gamma(\alpha+1)})^2} + \cdots\big)\nonumber\\
		&\sim s \Gamma(\alpha+1),\ \ \text{as $t\rightarrow\infty$}.
	\end{align}
	On substituting \eqref{pf4} in \eqref{pf1}, we get 
	\begin{align}\label{pf5}
		\mathbb{E}(e^{-s\hat{X}})&\sim \sum_{j=0}^ {\infty}\frac{\big(-s \Gamma(1+\alpha))^j}{\Gamma(j\alpha+1)} \nonumber\\
		&= E_{\alpha,1}(-s \Gamma(1+\alpha)).
	\end{align}
	It is known that (see Cahoy (2007), p. 60)
	\begin{equation}\label{pf3}
		E_{\alpha}(-c)
		= \alpha^{-1}\int_{0}^{\infty} e^{-c z}
		g_{\alpha}\!\left(z^{-1/\alpha}\right)\,
		z^{-1-1/\alpha}\,\mathrm{d}z ,
	\end{equation}
	where $0<\alpha<1$, $c \in \mathbb{R}$ and $g_{\alpha}(t)$ is the stable subordinators's pdf whose Laplace transform is $\mathcal{L}(g_{\alpha}(t);s)= e^{-s^{\alpha}}$. 
	
	Now, by using \eqref{pf3} in \eqref{pf5}, we get
	\begin{align*}
		\mathbb{E}(e^{-s\hat{X}}) 
		&\sim \alpha^{-1}\int_{0}^{\infty}
		\exp\!\left(-s\Gamma(\alpha+1)z\right)
		g_{\alpha}\!\left(z^{-\frac{1}{\alpha}}\right)
		z^{-1-\frac{1}{\alpha}}\,\mathrm{d}z\\
		&= \int_{0}^{\infty} e^{-sx}
		\frac{(\Gamma(\alpha+1))^{\frac{1}{\alpha}}}{\alpha}
		g_{\alpha}\!\left(
		\left(x/\Gamma(\alpha+1)\right)^{-\frac{1}{\alpha}}
		\right)
		x^{-1-\frac{1}{\alpha}}\,\mathrm{d}x ,
	\end{align*}
	where in the last step we have used the substitution $x=z\Gamma(\alpha+1)$.
	
	Therefore, for $t\to\infty$, the random variable $\hat{X}$ has a non-degenerate limiting distribution with the pdf 
	\begin{equation*}
		f_{\alpha}(x)
		=
		\frac{(\Gamma(\alpha+1))^{\frac{1}{\alpha}}}{\alpha}
		g_{\alpha}\!\left(
		\left(x/\Gamma(\alpha+1)\right)^{-\frac{1}{\alpha}}
		\right)
		x^{-1-\frac{1}{\alpha}},
		\ \ x\ge 0.
	\end{equation*}
	Its moments are given by
	\begin{align*}
		\mathbb{E}({\hat{X}}^r)
		&= \int_{0}^{\infty} x^r
		\frac{(\Gamma(\alpha+1))^{\frac{1}{\alpha}}}{\alpha}
		g_{\alpha}\!\left(
		\left(x/\Gamma(\alpha+1)\right)^{-\frac{1}{\alpha}}
		\right)
		x^{-1-\frac{1}{\alpha}}\,\mathrm{d}x \\
		&= (\Gamma(\alpha+1))^r
		\int_{0}^{\infty} y^{-r\alpha}
		g_{\alpha}(y)\,\mathrm{d}y,\ \  \text{substitute $y=\left(x/\Gamma(\alpha+1)\right)^{-\frac{1}{\alpha}}$} \\
		&= \frac{(\Gamma(\alpha+1))^r \Gamma(1+r)}
		{\Gamma(r\alpha+1)},
	\end{align*}
	where in the last step we have used the following Mellin transform of $g_\alpha(\cdot)$ (see Cahoy (2007), p. 84):
	\begin{equation*}
		\int_{0}^{\infty} t^r g_\alpha(t)\,\mathrm{d}t
		=
		\begin{cases}
			\dfrac{\Gamma(1-r/\alpha)}{\Gamma(1-r)}, & -\infty<r<\alpha,\\
			\infty, & r\ge \alpha .
		\end{cases}
	\end{equation*}

	\section{Compound Version of $\{{\hat{N}_\alpha(t)}\}_{t\ge 0}$}\label{section4}
	In this section, we discuss a compound version of $\{{\hat{N}_\alpha(t)}\}_{t\ge 0}$. It is defined as 
	\begin{equation}\label{comdef}
		\hat{Z}_{\alpha}(t) \coloneqq \sum_{i=1}^{\hat{N}_{\alpha}(t)} X_i, \  \ \text{$ t \ge 0$},
	\end{equation}
	where $\{X_i\}_{i \ge 1}$ is a sequence of positive integer valued i.i.d. random variables with finite mean and variance, and these are independent of $\{{\hat{N}_\alpha(t)}\}_{t\ge 0}$.


	Next, we obtain the system of differential equations that governs the state probabilities of $\{{\hat{Z}_\alpha(t)}\}_{t\ge 0}$.
	\begin{proposition}
		The state probabilities $\bar{p}_\alpha(m,t)
		= {\Pr}\{\hat{Z}_\alpha(t) = m \}$ satisfies the following system of differential equations:
		\begin{align*}
			&\frac{\mathrm{d}^{2\alpha}}{\mathrm{d}t^{2\alpha}}\bar{p}_{\alpha}(0,t)+2\lambda\frac{\mathrm{d}^{\alpha}}{\mathrm{d}t^{\alpha}}\bar{p}_{\alpha}(0,t)=-\lambda^{2}\bar{p}_{\alpha}(0,t), \\
			&\frac{\mathrm{d}^{2\alpha}}{\mathrm{d}t^{2\alpha}}\bar{p}_{\alpha}(m,t)+2\lambda\frac{\mathrm{d}^{\alpha}}{\mathrm{d}t^{\alpha}}\bar{p}_{\alpha}(m,t)=-\lambda^{2}\Big(\bar{p}_{\alpha}(m,t)-\sum_{j=1}^m{\Pr}\{X_1 = j\}\bar{p}_\alpha(m-j,t)\Big), \text{ for $m \ge 1$ }.
		\end{align*}
	\end{proposition}
	\begin{proof}
		Let us consider the notation
		$
		q_m := {\Pr}\{X_1 = m\}$ for all $ m \ge 1,$   $ q_m^{*n} := \Pr \{X_1 + \cdots + X_n = m \}$. Then $q_m^{*n} = 0$ for all integers $m< n$.
		Thus, 
		\begin{align}\label{compmf}
			\bar{p}_\alpha(m,t)
			= 
			\begin{cases}
				\hat{p}_{\alpha}(0,t), 
				& m = 0, \\[2mm]
				\displaystyle \sum_{n=1}^{m} q_m^{*n}\,
				\hat{p}_{\alpha}(n,t), 
				& m \ge 1 .
			\end{cases}
		\end{align}
		For $m=0 $, from \eqref{compmf}, we have
		\begin{align*}
			\frac{\mathrm{d}^{2\alpha}}{\mathrm{d}t^{2\alpha}}\bar{p}_{\alpha}(0,t)+2\lambda\frac{\mathrm{d}^{\alpha}}{\mathrm{d}t^{\alpha}}\bar{p}_{\alpha}(0,t)
			&=\frac{\mathrm{d}^{2\alpha}}{\mathrm{d}t^{2\alpha}}\hat{p}_{\alpha}(0,t)+2\lambda\frac{\mathrm{d}^{\alpha}}{\mathrm{d}t^{\alpha}}\hat{p}_{\alpha}(0,t)\\&=-\lambda^{2}\hat{p}_{\alpha}(0,t),\ \text{(using \eqref{diff1})}\\&=-\lambda^{2}\bar{p}_{\alpha}(0,t).
		\end{align*}
		Thus, the result holds true for $m=0$. Now for all integers $m \ge 1$, from \eqref{compmf} we have
		\begin{align}\label{B}
			\frac{\mathrm{d}^{2\alpha}}{\mathrm{d}t^{2\alpha}}\bar{p}_{\alpha}(m,t)+2\lambda\frac{\mathrm{d}^{\alpha}}{\mathrm{d}t^{\alpha}}\bar{p}_{\alpha}(m,t)&= \sum_{n=1}^m q_m^{*n}\Big(\frac{\mathrm{d}^{2\alpha}}{\mathrm{d}t^{2\alpha}}\hat{p}_{\alpha}(n,t)+2\lambda\frac{\mathrm{d}^{\alpha}}{\mathrm{d}t^{\alpha}}\hat{p}_{\alpha}(n,t)\Big) \nonumber\\&=-\lambda^2 \sum_{n=1}^m q_m^{*n}   \big(\hat{p}_{\alpha}(n,t) -\hat{p}_{\alpha}(n-1,t) \big), \ \text{(using \eqref{diff1})}\nonumber \\
			&= -\lambda^2 \sum_{n=1}^m q_m^{*n} \hat{p}_{\alpha}(n,t)+ \lambda^2 \sum_{n=1}^m \Big( \sum_{j=1}^m q_{m-j}^{*(n-1)}q_j \Big) \hat{p}_{\alpha}(n-1,t) \nonumber \\
			&= -\lambda^2 \bar{p}_\alpha(m,t) + \lambda^2 \sum_{j=1}^m q_j \sum_{n=1}^m q_{m-j}^{*(n-1)} \hat{p}_{\alpha}(n-1,t).
		\end{align}
		Note that $q_h^{*0}= q_0^{*h}=1$ if $h=0$ and $0$ if $h \ge 1 $. Also,
		\begin{equation}\label{+}
			\bar{p}_\alpha(m,t)
			= \sum_{n=1}^{m'} q_{m}^{*n}\,
			\hat{p}_{\alpha}(n,t),\  \text{ for $ m' \ge m $ }.
		\end{equation}
		Now,
		
		\begin{align*}
			\sum_{j=1}^m q_j \sum_{n=1}^m q_{m-j}^{*(n-1)} \hat{p}_{\alpha}(n-1,t)
			&= \sum_{j=1}^{m-1} q_j \sum_{n=1}^{m} q_{m-j}^{*(n-1)} \hat{p}_{\alpha}(n-1,t)+ q_m \sum_{n=1}^{m} q_0^{*(n-1)} \hat{p}_{\alpha}(n-1,t)\\
			&= \sum_{j=1}^{m-1} q_j \sum_{n=2}^{m} q_{m-j}^{*(n-1)} \hat{p}_{\alpha}(n-1,t)+ q_m \bar{p}_{\alpha}(0,t), \ \text{(using \eqref{compmf}})\\
			&= \sum_{j=1}^{m-1} q_j \sum_{l=1}^{m-1} q_{m-j}^{*l} \hat{p}_{\alpha}(l,t)+ q_m \bar{p}_{\alpha}(0,t)\\
			&= \sum_{j=1}^{m-1} q_j \bar{p}_{\alpha}(m-j,t)+ q_m \bar{p}_{\alpha}(0,t) , \ \text{(using \eqref{+}})\\&= \sum_{j=1}^m q_{j}\bar{p}_{\alpha}(m-j,t).
		\end{align*}
	The result follows on substituting the above equation in \eqref{B}.
		
	\end{proof}   
	The mean and variance of $\{\hat{Z}_{\alpha}(t)\}_{t\ge 0}$ can be obtained using Wald's identity as follows (see Mikosch (2009)):
	\begin{equation*}
		\mathbb{E}(\hat{Z}_\alpha(t)) = \mathbb{E}(\hat{N}_\alpha(t))\mathbb{E}(X_1) = \lambda^2 t^{2\alpha} E_{\alpha,2\alpha+1}(-2\lambda t^{\alpha})\mathbb{E}(X_1),\  \text{(using \eqref{mean})},
	\end{equation*}
	and
	\begin{align*}
		\operatorname{Var}(\hat{Z}_\alpha(t))&= \mathbb{E}(\hat{N}_\alpha(t))\operatorname{Var}(X_1) + \operatorname{Var}(\hat{N}_\alpha(t))\big(\mathbb{E}(X_1)\big)^2\\
		&= \lambda^2 t^{2\alpha} E_{\alpha,2\alpha+1}(-2\lambda t^{\alpha}) \operatorname{Var}(X_1) + \operatorname{Var}(\hat{N}_\alpha(t))\big(\mathbb{E}(X_1)\big)^2,
	\end{align*}
	where $\operatorname{Var}(\hat{N}_\alpha(t))$ is given in Remark \ref{comvar}.
	
	Next, we derive the mgf $\bar{M}_{\alpha}(u,t) = \mathbb{E}(e^{-u \hat{Z}_\alpha(t)})$, $u \ge 0$ of $\{{\hat{Z}_\alpha(t)}\}_{t\ge 0}$ as follows:
	\begin{align*}
		\bar{M}_{\alpha}(u,t) &= \mathbb{E}\left( \mathbb{E}\left( e^{-u\sum_{m=1}^{\hat{N}_\alpha(t)} X_m} \,\middle|\, \hat{N}_\alpha(t) \right) \right) \\
		&= \mathbb{E}\left( \left( \mathbb{E}\left( e^{-uX_1} \right) \right)^{\hat{N}_\alpha(t)} \right) \\
		&=\frac{\sqrt{r(u)}+1}{2\sqrt{r(u)}} 
		E_{\alpha,1}\left(-\lambda(1-\sqrt{r(u)}t^{\alpha}\right)
		+ \frac{\sqrt{r(u)}-1}{2\sqrt{r(u)}} 
		E_{\alpha,1}\left(-\lambda(1+\sqrt{r(u)}t^{\alpha}\right),
	\end{align*}
	where $r(u)=\mathbb{E}(e^{- u X_1})$ and the last step follows from \eqref{pgf}.
	
	\subsection{Ruin probabilities}
	Begin and Macci (2013) introduced and studied the following  fractional risk process:
	\begin{equation*}
		R_{\alpha}(t) = z + ct - \sum_{i=1}^{N_{\alpha}^{h}(t)} X_i,
		\  \ t \ge 0,
	\end{equation*}
	where $z > 0$ is the initial capital and $c > 0$ is the constant premium rate. Here, the claim numbers received in the insurance company is modeled  by the renewal process $\{N_{\alpha}^{h}(t)\}_{t \ge 0}$  whose interarrival times density is given by
	\begin{equation*}
		f_{\alpha}^{h}(t)
		=
		\lambda^{h} t^{\alpha h - 1}
		E_{\alpha, \alpha h}^{h}\!\left(-\lambda t^{\alpha}\right),
		\ \  \lambda > 0,\   t > 0.
	\end{equation*}
	Also, the i.i.d. positive random variables $\{X_i\}_{i \ge 1}$'s represent the claim sizes and these are independent with $\{N_{\alpha}^{h}(t)\}_{t\ge 0}$. For  $h=1$, the renewal process $\{N_{\alpha}^{h}(t)\}_{t\ge 0}$  reduces to TFPP $\{{N}_{\alpha}(t)\}_{t\ge 0}$ and $\{{R}_{\alpha}(t)\}_{t\ge 0}$ reduces to a risk process studied in Biard and Saussereau (2014), Constantinescu {\it et al.} (2019), {\it etc}.
	Note that for $h=2$, the renewal process $\{N_{\alpha}^{h}(t)\}_{t\ge 0}$  reduces to $\{\hat{N}_{\alpha}(t)\}_{t\ge 0}$. Thus, using the $\{\hat{N}_{\alpha}(t)\}_{t\ge 0}$ as a claim number process, we introduce the following risk process: 
	\begin{equation}\label{ruin}
		\hat{R}_{\alpha}(t) = z + ct - \sum_{i=1}^{\hat{N}_{\alpha}(t)} X_i, \ \  t \ge 0,
	\end{equation}
	where  $\{X_i\}_{i \ge 1}$ are i.i.d. positive random variables with finite mean  and finite variance. 
	
	The expected value of $\{\hat{R}_{\alpha}(t)\}_{t\ge 0}$ is given by
	\begin{equation*}
		\mathbb{E}(\hat{R}_{\alpha}(t))= z+ ct -\mathbb{E}(X_1) \lambda^2 t^{2\alpha} E_{\alpha,2\alpha+1}(-2\lambda t^{\alpha}).
	\end{equation*}
	Next, we derive the covariance of $\{\hat{R}_{\alpha}(t)\}_{t\ge 0}$ as follows:
	
	Let $ 0< s\le t$, then 
	\begin{align}\label{C}
		&\operatorname{Cov}\big({\hat{R}_{\alpha}(s)},{\hat{R}_{\alpha}(t)}\big) 
		=\operatorname{Cov}\Big(\sum_{i=1}^{\hat{N}_{\alpha}(s)} X_i, \sum_{j=1}^{\hat{N}_{\alpha}(t)} X_j\Big) \nonumber \\
		&= \mathbb{E}\Big(\sum_{i=1}^{\infty} \sum_{j=1}^{\infty} X_i X_j \mathbb{I}\{\hat{N}_{\alpha}(s) \geq j, \hat{N}_{\alpha}(t) \geq i\}\Big) - (\mathbb{E} X_1)^2 \mathbb{E}(\hat{N}_{\alpha}(t)) \mathbb{E}(\hat{N}_{\alpha}(s)) \nonumber \\
		&= \mathbb{E}\Big(\sum_{j=1}^{\infty} X_j^2 \mathbb{I}\{\hat{N}_{\alpha}(s) \geq j\}\Big) + \mathbb{E}\Big(\mathop{\sum  \sum}_{i\neq j}  X_i X_j \mathbb{I}\{\hat{N}_{\alpha}(s) \geq j, \hat{N}_{\alpha}(t) \geq i\}\Big) \nonumber \\
		&\ - (\mathbb{E} X_1)^2 \mathbb{E}(\hat{N}_{\alpha}(t)) \mathbb{E}(\hat{N}_{\alpha}(s)) \nonumber \\
		&= \mathbb{E}(X_1^2) \sum_{j=1}^{\infty} \Pr \{\hat{N}_{\alpha}(s) \geq j\} - (\mathbb{E} X_1)^2 \mathbb{E}(\hat{N}_{\alpha}(t)) \mathbb{E}(\hat{N}_{\alpha}(s)) \nonumber \\
		&\ + (\mathbb{E} X_1)^2 \Big(\sum_{i=1}^{\infty} \sum_{j=1}^{\infty} \Pr \{\hat{N}_{\alpha}(s) \geq j, \hat{N}_{\alpha}(t) \geq i\} - \sum_{j=1}^{\infty} \Pr \{\hat{N}_{\alpha}(s) \geq j\}\Big) \nonumber \\
		&= \operatorname{Var}(X_1) \mathbb{E} (\hat{N}_{\alpha}(s))+ (\mathbb{E} X_1)^2 \operatorname{Cov}(\hat{N}_{\alpha}(s), \hat{N}_{\alpha}(t)) \nonumber \\
		&= \operatorname{Var}(X_1) \lambda^2 s^{2\alpha} E_{\alpha,2\alpha+1}(-2\lambda s^{\alpha}) + (\mathbb{E} X_1)^2 \operatorname{Cov}(\hat{N}_{\alpha}(s), \hat{N}_{\alpha}(t)),
	\end{align}
	where in the last step we have used \eqref{mean}. Thus, the variance of $\{\hat{R}_{\alpha}(t)\}_{t\ge 0}$ can be obtained by substituting $s=t$ in \eqref{C}  and it is given by
	
	\begin{equation*}
		\operatorname{Var}(\hat{R}_{\alpha}(t))
		=\operatorname{Var}(X_1) \lambda^2 t^{2\alpha} E_{\alpha,2\alpha+1}(-2\lambda t^{\alpha}) + (\mathbb{E} X_1)^2 \operatorname{Var}(\hat{N}_{\alpha}(t)),
	\end{equation*}
	where $ \operatorname{Var}(\hat{N}_{\alpha}(t))$ is given in Remark \ref{comvar}.

	Next, we discuss some ruin probability related results for $\{\hat{R}_{\alpha}(t)\}_{t\ge 0}$ when claim size follows exponential and subexponential  distributions.
	
	\subsubsection{Exponentially distributed claim sizes}
	Let the claim sizes $\{X_i\}_{i \ge 1}$ in \eqref{ruin} be exponentially distributed with parameter $\theta> 0$ . Its ruin time $T$ is defined as  
	$$
	T \coloneqq \inf\{t > 0: \hat{R}_{\alpha}(t) < 0\}.
	$$
	Note that  for all $t$ if $\hat{R}_{\alpha}(t) \ge 0$, then $T = \infty$  . 
	
	Using Theorem 1 of Borokov and Dickson (2008), the pdf $f_T(t)$ of the ruin time $T$ is given by
	$$
	\hat{f}_T(t) = e^{-\theta(z+ct)} \sum_{n=0}^\infty \frac{\theta^n (z+ct)^{n-1}}{n!} \left( z + \frac{ct}{n+1} \right) (f_{\hat{T}_1^{\alpha}})^{*(n+1)}(t),
	$$
	where $(f_{\hat{T}_1^{\alpha}})^{*(n+1)}(t)$ denotes the $(n+1)$-fold convolution of the function $f_{\hat{T}_1^{\alpha}}(t)$ given in \eqref{renpmf}.
	
	The ruin probability $\hat{\psi}_z(t)$ of $\hat{R}_{\alpha}(t)$ in finite time $t$ is defined by $\hat{\psi}_z(t) = \mathrm{Pr}\{T \le t < \infty\} 
	$.
	From Theorem 1 of Malinovskii (1998) and \eqref{renlap}, the Laplace transform  of $\hat{\psi}_z(t)$ can be obtained as
	$$
	s\int_0^\infty e^{-st}\hat{\psi}_z(t)\mathrm{d}t= y(s) \exp\{-z\theta(1-y(s))\},\ \ s > 0,
	$$
	where $y(s)$ is a solution of
	\begin{equation*}
		y(s) = \lambda^2/ \big(\lambda+(s+c\theta(1-y(s)))^{\alpha}\big)^2.  
	\end{equation*}

	\subsubsection{Ruin probability for subexponentially distributed claim sizes}
	Let us assume that the distribution function $F_{X_1}(t) = \mathrm{Pr}\{X_1 \le t\}$ of the claims sizes in the model \eqref{ruin} be subexponential, that is, $\lim_{t\to\infty} (1 - F_{X_1}^{*2}(t)) / (1 - F_{X_1}(t)) = 2$. 
	Here, we derive an asymptotic behavior of the finite-time ruin probability for subexponential claim sizes when the initial capital becomes arbitrarily large, that is, $z \to \infty$.
	
	The following Lemma related to the  subexponential distribution will be used (see Asmussen and Albrecher (2010), Lemma X.2.2):
	\begin{lemma}\label{lemma5}
		Let $N$ be an integer valued random variable with $\mathbb{E}(u^N) < \infty$ for some $u > 1$ and it is independent with a sequence of i.i.d. random variables $\{Y_i\}_{i\ge 1}$ with common subexponential distribution $\bar{F}_{Y_1}(t)= \mathrm{Pr}\{Y_1 > t\}$ . Then,  
		$$
		\mathrm{Pr}\Big\{\sum_{i=1}^N Y_i > t\Big\} \sim \mathbb{E}(N) \bar{F}_{Y_1}(t), \ \ \text{as $t \to \infty$}.
		$$
	\end{lemma}

	\begin{proposition}\label{prop2}
		Let the distribution function of the claim sizes in the risk model \eqref{ruin}  be subexponential. Then 
		\begin{align*}
			\hat{\psi}_z(t)\sim \lambda^2 t^{2\alpha} E_{\alpha,2\alpha+1}(-2\lambda t^{\alpha}) \bar{F}_{X_1}(z),\ \ \text{as } z \to \infty.
		\end{align*}
	\end{proposition}
	
	\begin{proof} Note that
		$$
		\Big\{ \sum_{i=1}^{\hat{N}_{\alpha}(t)} X_i > z + ct \Big\} \subseteq \Big\{ \sum_{i=1}^{\hat{N}_{\alpha}(t^*)} X_i > z + ct^* \text{ for some } t^* \le t < \infty \Big\} \subseteq \Big\{ \sum_{i=1}^{\hat{N}_{\alpha}(t)} X_i > z \Big\}.
		$$
		Thus, the following inequalities hold for the finite time ruin probability of $\{\hat{R}_{\alpha}(t)\}_{t \ge 0}$ :
		\begin{equation*}
			\mathrm{Pr}\Big\{ \sum_{i=1}^{\hat{N}_{\alpha}(t)} X_i > z + ct \Big\} \le \mathrm{Pr}\Big\{ \sum_{i=1}^{\hat{N}_{\alpha}(t^*)} X_i > z + ct^* \text{ for some } t^* \le t < \infty \Big\}
			\le \mathrm{Pr}\Big\{ \sum_{i=1}^{\hat{N}_{\alpha}(t)} X_i > z \Big\}.
		\end{equation*}
		Now,  dividing the above equation by $\mathrm{Pr}\big\{ \sum_{i=1}^{\hat{N}_{\alpha}(t)} X_i > z \big\}$ and taking the limit $z \to \infty$, we get
		\begin{equation}\label{0}
			\mathrm{Pr}\Big\{ \sum_{i=1}^{\hat{N}_{\alpha}(t^*)} X_i > z + ct^* \text{ for some } t^* \le t < \infty \Big\} \sim \mathrm{Pr}\Big\{ \sum_{i=1}^{\hat{N}_{\alpha}(t)} X_i > z \Big\},
		\end{equation}
		where we have used the fact that $\lim_{z \rightarrow \infty }{\bar{F}_{X_1}(z+ct)}/{\bar{F}_{X_1}(z)}=1$.
		Note that the  pgf  of $\{{\hat{N}_\alpha(t)}\}_{t\ge 0}$  given in \eqref{pgf} is finite for some $u> 1$. Thus, on using the Lemma \ref{lemma5} in \eqref{0}, we get 
		\begin{align*}
			\hat{\psi}_z(t) = \mathrm{Pr}\Big\{ \sum_{i=1}^{\hat{N}_{\alpha}(t^*)} X_i > z + ct^* \text{ for some } t^* \le t < \infty \Big\}
			&\sim \mathrm{Pr}\Big\{ \sum_{i=1}^{\hat{N}_{\alpha}(t)} X_i > z \Big\}\\
			&\sim \mathbb{E}\big(\hat{N}_{\alpha}(t)\big) \bar{F}_{X_1}(z), \ \ \text{as } z \to \infty.
		\end{align*}
		On substituting \eqref{mean} in the above equation, we get the required result.
	\end{proof}
	
	Next, motivated from the Proposition 5 of Biard and Saussereau (2014), we extend the Proposition \ref{prop2} for the case of $k$-dimensional risk processes.
	
	Consider the following independent risk processes:
	\begin{equation*}
		{\hat{R}}^{(j)}_{\alpha}(t) = {z}^{(j)} + {c}^{(j)}t - \sum_{i=1}^{\hat{N}_{\alpha}(t)} {X}^{(j)}_i, \ \ j= 1,2 \dots k,
	\end{equation*}
	which can be rewritten as follows:
	\begin{equation}\label{vec}
		\bar{\hat{R}}_{\alpha}(t) = \bar{z} + \bar{c}t - \sum_{i=1}^{\hat{N}_{\alpha}(t)} \bar{X}_i,\ \ t \ge 0,
	\end{equation}
	where $\bar{\hat{R}}_{\alpha}(t) = \big(\hat{R}_{\alpha}^{(1)}(t), \hat{R}_{\alpha}^{(2)}(t), \ldots, \hat{R}_{\alpha}^{(k)}(t)\big)
	$ is an $k$-dimensional risk process. Here $\bar{z} = (z^{(1)}, z^{(2)}, \ldots, z^{(k)})$ is the initial capital vector, $\bar{c} = (c^{(1)}, c^{(2)}, \ldots, c^{(k)})$ is the premium intensity vector, and $\bar{X}_i = (X_i^{(1)}, X_i^{(2)}, \ldots, X_i^{(k)})$ is the $i$-th claim size vector. Moreover, $\{\bar{X}_i\}_{i\ge 1}$ is a sequence of i.i.d. random vectors with the following joint distribution function:
	\begin{align*}
		F(x_1, x_2, \ldots, x_k) &= \mathrm{Pr}\Big\{X^{(1)} \le x_1, X^{(2)} \le x_2, \ldots, X^{(k)} \le x_k\Big\}\\&= \prod_{j=1}^k \mathrm{Pr}\{X^{(j)} \le x_j\}\\&= \prod_{j=1}^k F^{(j)}(x_j).
	\end{align*}
	
	Next, we obtain an extension of the Proposition \ref{prop2} for the risk model defined in \eqref{vec} by following the same lines as the proof of Proposition \ref{prop2}.
	\begin{proposition}
		Assume that the distributions $F^{(j)}$, $j=1, 2, \ldots, k$ of the claim sizes are subexponentially distributed. For an initial capital vector $\bar{z}$, let
		$$
		\tau_{\max}(\bar{z}) = \inf\big\{s > 0: \max\{\hat{R}_{\alpha}^{(1)}(s), \hat{R}_{\alpha}^{(2)}(s), \ldots, \hat{R}_{\alpha}^{(k)}(s)\} < 0\big\},
		$$
		be the first time all the components of $\bar{\hat{R}}_{\alpha}(t)$ are negative. Then, for any $t > 0$, it holds that
		$$
		\mathrm{Pr}\{\tau_{\max}(\bar{z}) \le t\} \sim \mathbb{E}(\hat{N}_{\alpha}(t))^k \prod_{j=1}^k\left(1-F^{(j)}(z_j)\right),
		$$
		as $z_j \to \infty$ for all $j=1, 2, \ldots, k$.
	\end{proposition}
	
	\section{$\{\hat{N}_\alpha(t)\}_{t \geq 0}$ time-changed by a L\'evy Subordinator and its inverse}\label{sec5}
	In this section, we consider a time-changed version of the process $\{\hat{N}_\alpha(t)\}_{t \geq 0}$ where the time-changes is done by an independent L\'evy subordinator $\{D_{f}(t)\}_{t\ge0}$ whose moments are finite, that is, $\mathbb{E}(D^r_{f}(t))<\infty$ for all $r>0$. We denote the obtained process by $\{\hat{N}_\alpha^{f}(t)\}_{t\ge0}$.
	Thus, 
	\begin{equation}\label{tc}
		\hat{N}_\alpha^{f}(t)\coloneqq\hat{N}_\alpha(D_{f}(t)),\ \ t\ge0,
	\end{equation}
	where $\{\hat{N}_\alpha(t)\}_{t\ge0}$ is independent of $\{D_{f}(t)\}_{t\ge0}$.
	For $\alpha=1$, the process $\{\hat{N}^f_\alpha(t)\}_{t \geq 0}$ reduces to a time-changed version of $\{\hat{N}(t)\}_{t \geq 0}$, that is, 
	\begin{equation*}
		\hat{N}^{f}(t)\coloneqq\hat{N}(D_{f}(t)),\ \ t\ge0.
	\end{equation*}
	\begin{theorem}\label{5.1}
		The pmf $\hat{p}_\alpha^f(n,t)=\Pr\{\hat{N}_\alpha^{f}(t)=n\},\ n \ge 0$ of $\{\hat{N}^f_\alpha(t)\}_{t \geq 0}$ is given by 
		\begin{align}\label{pmftcl}
			\hat{p}_\alpha^f(n,t) &= \frac{\lambda^{2n}}{(2n)!}\sum_{k=0}^{\infty}\frac{(2n+k)!}{k!}  \frac{(-\lambda)^k}{\Gamma(\alpha(2n+k)+1)} \mathbb{E}\big(D_f^{\alpha(2n+k)}(t)\big) \nonumber \\ & \ \  + \frac{\lambda ^{2n+1}}{(2n+1)!}\sum_{k=0}^{\infty}\frac{(2n+k+1)!}{k!}  \frac{(-\lambda)^k}{\Gamma(\alpha(2n+k+1)+1)}  \mathbb{E}\big(D_f^{\alpha(2n+k+1)}(t)\big).
		\end{align}
	\end{theorem}
	\begin{proof}
		Let $g_f(y, t)$ be the density of $\{D_{f}(t)\}_{t\ge0}$. From \eqref{fracpmf} and \eqref{tc}, we have
		\begin{align*}
			\hat{p}_\alpha^f(n,t)&=\int_0^{\infty} \hat{p}_\alpha(n,y) g_f(y, t) \mathrm{d}y \\
			&=\int_0^{\infty}\Big(\lambda^{2n}y^{2n\alpha} {E}_{\alpha,2n\alpha+1}^{2n+1}(-\lambda  y^{\alpha})+\lambda^{2n+1}y^{(2n+1)\alpha} {E}_{\alpha,(2n+1)\alpha+1}^{2n+2}(-\lambda y^{\alpha})\Big)g_f(y, t) \mathrm{d}y\\
			&=\int_0^{\infty} \sum_{k=0}^{\infty}  \bigg( \frac{\lambda ^{2n}}{(2n)!} \frac{(2n+k)!}{k!} \frac{(-\lambda )^k y^{\alpha(2n+k)}}{\Gamma(\alpha(2n+k)+1)}\\
			&\ \ + \frac{\lambda^{2n+1}}{(2n+1)!} \frac{(2n+k+1)!}{k!} \frac{(-\lambda )^k y^{\alpha(2n+k+1)}}{\Gamma(\alpha(2n+k+1)+1)} \bigg)g_f(y, t) \mathrm{d}y\\
			&=  \frac{\lambda^{2n}}{(2n)!}\sum_{k=0}^{\infty}\frac{(2n+k)!}{k!}  \frac{(-\lambda)^k}{\Gamma(\alpha(2n+k)+1)} \int_0^{\infty}  y^{\alpha(2n+k)}g_f(y, t) \mathrm{d}y \\ & \ \  + \frac{\lambda ^{2n+1}}{(2n+1)!}\sum_{k=0}^{\infty}\frac{(2n+k+1)!}{k!}  \frac{(-\lambda)^k}{\Gamma(\alpha(2n+k+1)+1)} \int_0^{\infty}y^{\alpha(2n+k+1)} g_f(y, t) \mathrm{d}y.
		\end{align*}
		Thus, the proof follows.
	\end{proof}
	
	\begin{remark}
		For $\alpha=1$, we get the pmf $\hat{p}^{f}(n,t)=\mathrm{Pr}\{\hat{N}^{f}(t)=n\}$, $n \ge 0$ of $\{\hat{N}^{f}(t)\}_{t\ge0}$ in the following form:
		\begin{align*}
			\hat{p}^{f}(n,t)=& \frac{\lambda^{2n}}{(2n)!}\sum_{k=0}^{\infty}\frac{(2n+k)!}{k!}  \frac{(-\lambda)^k}{\Gamma((2n+k)+1)}  \mathbb{E}\big(D_f^{(2n+k)}(t)\big) \\ & \ \  + \frac{\lambda ^{2n+1}}{(2n+1)!}\sum_{k=0}^{\infty}\frac{(2n+k+1)!}{k!}  \frac{(-\lambda)^k}{\Gamma((2n+k+1)+1)}  \mathbb{E}\big(D_f^{(2n+k+1)}(t)\big)\\
			&= \frac{\lambda^{2n}}{(2n)!}\int _0^\infty e^{-\lambda y}y^{2n}g_f(y,t)\mathrm{d}y  + \frac{\lambda ^{2n+1}}{(2n+1)!} \int _0^\infty  e^{-\lambda y}y^{2n+1}g_f(y,t)\mathrm{d}y \\
			&=\frac{\lambda^{2n}}{(2n)!}\mathbb{E}\left(D_{f}^{2n}(t)e^{-\lambda D_{f}(t)}\right)+\frac{\lambda^{2n+1}}{(2n+1)!}\mathbb{E}\left(D_{f}^{2n+1}(t)e^{-\lambda D_{f}(t)}\right).				
		\end{align*}

		An alternative expression of the pmf $\hat{p}^{f}(n,t)$ can be obtained as follows:
		\begin{align*}
			\hat{p}^{f}(n,t)&=\int_0^{\infty} \hat{p}(n,s)\Pr\{D_f(t) \in \mathrm{d}s\}\\
			&=\int_0^{\infty} e^{-\lambda s} \Big(\frac{(\lambda s)^{2n}}{(2n)!} + \frac{(\lambda s)^{2n+1}}{(2n+1)!} \Big)\Pr\{D_f(t) \in \mathrm{d}s\}, \ \text{(using \eqref{pmf2pp})}\\
			&= \frac{(-1)^{2n}}{(2n)!} \frac{\mathrm{d}^{2n}}{\mathrm{d}u^{2n}} \left. \int_0^{\infty} e^{-\lambda s u} \Pr\{D_f(t) \in \mathrm{d}s\} \right|_{u=1}\\
			&\ \ +\frac{(-1)^{2n+1}}{(2n+1)!} \frac{\mathrm{d}^{2n+1}}{\mathrm{d}u^{2n+1}} \left. \int_0^{\infty} e^{-\lambda s u} \Pr\{D_f(t) \in\mathrm{d}s\} \right|_{u=1}\\
			&=\frac{1}{(2n)!} \frac{\mathrm{d}^{2n}}{\mathrm{d}u^{2n}} e^{-t f(\lambda u)} \Big|_{u=1}- \frac{1}{(2n+1)!} \frac{\mathrm{d}^{2n+1}}{\mathrm{d}u^{2n+1}} e^{-t f(\lambda u)} \Big|_{u=1}.
		\end{align*}
	\end{remark}
	\begin{theorem}\label{5.2}
		The pgf $\hat{G}_{\alpha}^{f}(u,t)=\mathbb{E}\left(u^{\hat{N}_{\alpha}^{f}(t)}\right)$, $|u|\le1$ of $\{\hat{N}_{\alpha}^{f}(t)\}_{t\ge0}$  is given by 
		{\small  \begin{equation}\label{pp}
				\hat{G}_{\alpha}^{f}(u,t) =\frac{\sqrt{u}+1}{2\sqrt{u}}\sum_{l=0}^{\infty}\frac{(-\lambda (1-\sqrt{u}))^l}{\Gamma(\alpha l+1)}\mathbb{E}\big(D^{\alpha l }_f(t)\big)+\frac{\sqrt{u}-1}{2\sqrt{u}}\sum_{l=0}^{\infty}\frac{(-\lambda (1+\sqrt{u}))^l}{\Gamma(\alpha l+1)}\mathbb{E}\big(D^{\alpha l }_f(t)\big).
		\end{equation}}
	\end{theorem}
	\begin{proof}
		From \eqref{pgf} and \eqref{tc}, we have
		\begin{align*}
			\hat{G}_{\alpha}^{f}(u,t)&=\int_{0}^{\infty}\hat{G}_\alpha(u,y)g_{f}(y,t)\mathrm{d}y\nonumber\\
			&=\int_{0}^{\infty}\Big(\frac{\sqrt{u}+1}{2\sqrt{u}}E_{\alpha,1}(-\lambda (1-\sqrt{u})y^{\alpha})+\frac{\sqrt{u}-1}{2\sqrt{u}}E_{\alpha,1}(-\lambda (1+\sqrt{u})y^{\alpha})\Big)g_{f}(y,t)\mathrm{d}y\nonumber\\
			&=\frac{\sqrt{u}+1}{2\sqrt{u}}\sum_{l=0}^{\infty}\frac{(-\lambda (1-\sqrt{u}))^l}{\Gamma(\alpha l+1)}\int_{0}^{\infty}y^{\alpha l }g_{f}(y,t)\mathrm{d}y\\
			&\ \ +\frac{\sqrt{u}-1}{2\sqrt{u}}\sum_{l=0}^{\infty}\frac{(-\lambda (1+\sqrt{u}))^l}{\Gamma(\alpha l+1)}\int_{0}^{\infty}y^{\alpha l }g_{f}(y,t)\mathrm{d}y,
		\end{align*}
		which gives the required result.
	\end{proof}

	
	\begin{remark}
		For $\alpha=1$, we get the pgf $\hat{G}^{f}(u,t)=\mathbb{E}\big(u^{\hat{N}^{f}(t)}\big)$ of $\{\hat{N}^{f}(t)\}_{t\ge0}$ in the following form:
		\begin{align*}
			\hat{G}^{f}(u,t)&=\frac{\sqrt{u}+1}{2\sqrt{u}}\sum_{l=0}^{\infty}\frac{(-\lambda (1-\sqrt{u}))^l}{\Gamma(l+1)}\mathbb{E}\big(D^{l }_f(t)\big)+\frac{\sqrt{u}-1}{2\sqrt{u}}\sum_{l=0}^{\infty}\frac{(-\lambda (1+\sqrt{u}))^l}{\Gamma( l+1)}\mathbb{E}\big(D^{ l }_f(t)\big)\nonumber\\
			&=\frac{\sqrt{u}+1}{2\sqrt{u}}\int _0^\infty e^{-\lambda (1-\sqrt{u})y}g_f(y,t)\mathrm{d}y+\frac{\sqrt{u}-1}{2\sqrt{u}}\int _0^\infty e^{-\lambda (1+\sqrt{u})y}g_f(y,t)\mathrm{d}y\nonumber\\
			&=\frac{\sqrt{u}+1}{2\sqrt{u}}e^{-tf(\lambda (1-\sqrt{u}))}+\frac{\sqrt{u}-1}{2\sqrt{u}}e^{-tf(\lambda (1+\sqrt{u}))},\ \ |u|\le1.		
		\end{align*}

		\noindent The  mean of $\{\hat{N}_\alpha^{f}(t)\}_{t\ge 0}$ is given by 
		\begin{equation*}
			\mathbb{E}(\hat{N}_\alpha^{f}(t))=\frac{\partial}{\partial u}\hat{G}_{\alpha}^{f}(u,t)\Big|_{u=1}=-\frac{1}{4}+\frac{\lambda \mathbb{E}(D^\alpha_f(t))}{2\Gamma(\alpha+1)}+\sum_{l=0}^{\infty} \frac{\mathbb{E}\big(D^{\alpha l }_f(t)\big)(-2\lambda)^l}{4\Gamma(\alpha l+1)}.
		\end{equation*}
		Its variance can be obtained as follows:
		\begin{align*}
			\mathbb{E}(\hat{N}_\alpha^{f}(t)(\hat{N}_\alpha^{f}(t)-1))&=\frac{\partial^2}{\partial u^2}\hat{G}_{\alpha}^{f}(u,t)\Big|_{u=1}\\&=\frac{3}{8}-\frac{\lambda \mathbb{E}(D^\alpha_f(t))}{2\Gamma(\alpha+1)}+\frac{\lambda^2\mathbb{E}(D^{2\alpha}_f(t))}{2\Gamma(2\alpha+1)}-\sum_{l=1}^{\infty} \frac{\lambda l (-2\lambda)^{l-1}\mathbb{E}\big(D^{\alpha l }_f(t)\big)}{4\Gamma(\alpha l+1)}\\& \ \ -\sum_{l=0}^{\infty} \frac{3 (-2\lambda)^l\mathbb{E}\big(D^{\alpha l }_f(t)\big)}{8\Gamma(\alpha l+1)}.
		\end{align*}
		Thus,
		\begin{align*}
			\operatorname{Var} (\hat{N}_\alpha^{f}(t))  = \frac{1}{8}+\frac{\lambda^2\mathbb{E}(D^{2\alpha}_f(t))}{2\Gamma(2\alpha+1)}&-\sum_{l=1}^{\infty} \frac{\lambda l (-2\lambda)^{l-1}\mathbb{E}\big(D^{\alpha l }_f(t)\big)}{4\Gamma(\alpha l+1)}-\sum_{l=0}^{\infty} \frac{ (-2\lambda)^l\mathbb{E}\big(D^{\alpha l }_f(t)\big)}{8\Gamma(\alpha l+1)}\nonumber\\
			&-\Bigg( -\frac{1}{4}+\frac{\lambda \mathbb{E}(D^\alpha_f(t))}{2\Gamma(\alpha+1)}+\sum_{l=0}^{\infty} \frac{\mathbb{E}\big(D^{\alpha l }_f(t)\big)(-2\lambda)^l}{4\Gamma(\alpha l+1)}\Bigg)^2.
		\end{align*}


	\end{remark}

	Next, we study the first passage times of $\{\hat{N}^{f}(t)\}_{t \ge 0}$, defined as
	\begin{equation*}
		\hat{T}_f^{n} := \inf\{ s \ge 0 : \hat{N}^f(s) \ge n \}.
	\end{equation*}
Its distribution can be written as
	\begin{equation*}
		\Pr\{\hat{T}_f^{n} < s\} = \Pr\{\hat{N}^f(s) \ge n\}
		= \sum_{m=n}^{\infty} \int_0^{\infty} e^{-\lambda x} \Big(\frac{(\lambda x)^{2m}}{(2m)!} + \frac{(\lambda x)^{2m+1}}{(2m+1)!} \Big)\Pr\{D_f(s) \in \mathrm{d}x\}.
	\end{equation*}
	Thus, its density function can be obtained as
	\begin{align}\label{its}
		\dfrac{\Pr\{\hat{T}_f^n \in \mathrm{d}s\}}{\mathrm{d}s}&=\frac{\mathrm{d}}{\mathrm{d}s} \sum_{m=n}^{\infty} \int_0^{\infty} e^{-\lambda x} \Big(\frac{(\lambda x)^{2m}}{(2m)!} + \frac{(\lambda x)^{2m+1}}{(2m+1)!} \Big)\Pr\{D_f(s) \in \mathrm{d}x\} \nonumber\\
		&=\frac{\mathrm{d}}{\mathrm{d}s}\int_0^\infty\Big(1-\sum_{l=0}^{n-1}\Big(\frac{(\lambda x)^{2l}}{(2l)!} + \frac{(\lambda x)^{2l+1}}{(2l+1)!} \Big)e^{-\lambda x}\Big)\Pr\{D_f(s) \in \mathrm{d}x\}\nonumber\\
		&=-\frac{\mathrm{d}}{\mathrm{d}s}\sum_{l=0}^{n-1}\frac{\lambda^{2l}}{(2l)!}\int _0^\infty x^{2l}e^{-\lambda x}\Pr\{D_f(s) \in \mathrm{d}x\} \nonumber\\
		&\ \ -\frac{\mathrm{d}}{\mathrm{d}s}\sum_{l=0}^{n-1}\frac{\lambda^{2l+1}}{(2l+1)!}\int _0^\infty x^{2l+1}e^{-\lambda x}\Pr\{D_f(s) \in \mathrm{d}x\}\nonumber\\
		&=-\frac{\mathrm{d}}{\mathrm{d}s}\sum_{l=0}^{n-1}\frac{(-\lambda)^{2l}}{(2l)!}\int _0^\infty\frac{\mathrm{d}^{2l}}{\mathrm{d} \lambda^{2l}}e^{-\lambda x}\Pr\{D_f(s) \in \mathrm{d}x\} \nonumber\\
		&\ \ -\frac{\mathrm{d}}{\mathrm{d}s}\sum_{l=0}^{n-1}\frac{(-\lambda)^{2l+1}}{(2l+1)!}\int _0^\infty\frac{\mathrm{d}^{2l+1}}{\mathrm{d} \lambda^{2l+1}}e^{-\lambda x}\Pr\{D_f(s) \in \mathrm{d}x\} \nonumber\\
		&=-\frac{\mathrm{d}}{\mathrm{d}s}\sum_{l=0}^{n-1}\frac{\lambda^{2l}}{(2l)!}\frac{\mathrm{d}^{2l}}{\mathrm{d} \lambda^{2l}}e^{-s f(\lambda)}+\frac{\mathrm{d}}{\mathrm{d}s}\sum_{l=0}^{n-1}\frac{\lambda^{2l+1}}{(2l+1)!}\frac{\mathrm{d}^{2l+1}}{\mathrm{d} \lambda^{2l+1}}e^{-s f(\lambda)},
	\end{align}
	where in the last step we have used \eqref{levy}.
	\begin{remark}
		In particular, from \eqref{its}, we get  the waiting time of the first event of $\{\hat{N}^f(t)\}_{t\ge 0}$ as
		\begin{equation*}
			\Pr\{\hat{T}_f^1 \in \mathrm{d}s\}= \big(f(\lambda) -\lambda f'(\lambda)+ \lambda s f(\lambda) f'(\lambda)\big)e^{-s f(\lambda)}\mathrm{d}s.
		\end{equation*}
		It coincides with the distribution of first passage times $ {T}_f^{n}= \inf\{ s \ge 0 : N(D_f(s)) \ge n \}$ of $\{N(D_f(t))\}_{t\ge 0}$ for $n=2$ (see Orsingher and Toaldo (2015), Eq. (3.4)).
		Also, a recursive relation between the distributions of $\hat{T}_f^n$ can be obtained as
		\begin{equation*}
			\Pr\{\hat{T}_f^n \in \mathrm{d}s\} = \Pr\{\hat{T}_f^{n-1} \in \mathrm{d}s\}
			-\frac{\mathrm{d}}{\mathrm{d}s}\frac{\lambda^{2n-2}}{(2n-2)!}\frac{\mathrm{d}^{2n-2}}{\mathrm{d} \lambda^{2n-2}}e^{-s f(\lambda)}\mathrm{d}s+\frac{\mathrm{d}}{\mathrm{d}s}\frac{\lambda^{2n-1}}{(2n-1)!}\frac{\mathrm{d}^{2n-1}}{\mathrm{d} \lambda^{2n-1}}e^{-s f(\lambda)}\mathrm{d}s.
		\end{equation*}
		
	\end{remark}

	Next, we discuss few examples of the time-changed process $\{\hat{N}^f_\alpha(t)\}_{t\ge 0}$ by considering some specific L\'evy subordinators such as gamma subordinator, tempered stable subordinator, and inverse Gaussian subordinator .
	\begin{example}
		($\{\hat{N}_\alpha(t)\}_{t\ge 0}$ time-changed by gamma subordinator) 
		Let $\{G(t)\}_{t\ge0}$ be a gamma subordinator with the following pdf:
		\begin{equation*}
			g(x,t)=\frac{a^{bt}}{\Gamma(bt)}x^{bt-1}e^{-ax},\ \  x>0,\ a>0,\ b>0.
		\end{equation*}
		Its rth moment is given by
		\begin{equation}\label{mga}
			\mathbb{E}(G^r(t))=\frac{\Gamma(bt+r)}{a^r\Gamma(bt)},\ \ r >0.
		\end{equation}
		Its associated   Bern\v stein function is $f_{1}(s)=b\log(1+s/a)$,  $s>0$ (see Applebaum (2009), p. 55). Its pdf satisfies the following differential equation (see Beghin and Vellaisamy (2018), Lemma 2.2):
		\begin{equation}\label{diffgam}
			\mathbb{D}_t^{\gamma}g(x,t)=b \mathbb{D}_t^{\gamma-1}\big(\log(a x)-\kappa(bt)\big)g(x,t),
		\end{equation}
		with  $g(x,0)=0$. Here, $\gamma\ge1$,  $\kappa(x)\coloneqq\Gamma^{\prime}(x)/\Gamma(x)$ is the digamma function and $\mathbb{D}_t^{\gamma}$ is the Riemann-Liouville fractional derivative defined in (\ref{RLd}).	
		
		On taking $f_1$ as a  Bern\v stein function in \eqref{tc}, we obtain 
		\begin{equation}\label{fgam}
			\hat{N}_\alpha^{f_1}(t)\coloneqq \hat{N}
			_\alpha (Z(t)),\ \ t\ge0.
		\end{equation}
		
		\begin{proposition}
			Let $\gamma\ge1$ and $\kappa(x)$ be the digamma function. Then, the pmf $\hat{p}_\alpha^{f_{1}}(n,t)=\mathrm{Pr}\{\hat{N}_\alpha^{f_1}(t)=n\}$, $n \ge 0$ solves the following fractional differential equation:
			\begin{equation*}
				\mathbb{D}_t^{\gamma}\hat{p}_\alpha^{f_{1}}(n,t)=b \mathbb{D}_t^{\gamma-1}\big(\log (a)-\kappa(bt)\big)\hat{p}_\alpha^{f_{1}}(n,t)+b\int_{0}^{\infty}\hat{p}_\alpha(n,x)\log (x) \mathbb{D}_t^{\gamma-1}g(x,t)\,\mathrm{d}x,
			\end{equation*}
			where $\hat{p}_\alpha(n,x)$ is the pmf of $\{\hat{N}_\alpha^{f}(t)\}_{t\ge0}$ given in \eqref{fracpmf}.
		\end{proposition}
		\begin{proof}
			From (\ref{fgam}), we have
			\begin{equation}\label{qaza122}
				\hat{p}_\alpha^{f_{1}}(n,t)=\int_{0}^{\infty}\hat{p}_\alpha(n,x)g(x,t)\,\mathrm{d}x.
			\end{equation}
			By taking the Riemann-Liouville fractional derivative on the both side of (\ref{qaza122}) and using \eqref{diffgam}, we get
			\begin{align*}
				\mathbb{D}_t^{\gamma}\hat{p}^{f_{1}}_{\alpha}(n,t)&=\int_{0}^{\infty}\hat{p}_\alpha(n,x)\mathbb{D}_t^{\gamma}g(x,t)\,\mathrm{d}x\\
				&=b\int_{0}^{\infty}\hat{p}_\alpha(n,x)\mathbb{D}_t^{\gamma-1}\big(\log(a x)-\kappa(bt)\big)g(x,t)\,\mathrm{d}x \\
				&=b \mathbb{D}_t^{\gamma-1}\log (a) \int_{0}^{\infty}\hat{p}_\alpha(n,x)g(x,t)\,\mathrm{d}x+b \int_{0}^{\infty} \hat{p}_\alpha(n,x)\log (x)\mathbb{D}_t^{\gamma-1}g(x,t)\,\mathrm{d}x\\
				&\ \ - b \mathbb{D}_t^{\gamma-1}\kappa(bt)\int_{0}^{\infty}\hat{p}_\alpha(n,x)g(x,t)\,\mathrm{d}x.
			\end{align*}
			The proof is completed on using (\ref{qaza122}).
		\end{proof}
		From \eqref{pmftcl} and \eqref{mga}, the pmf of  $\{\hat{N}_\alpha^{f_1}(t)\}_{t \geq 0}$  is obtained in the following form:
		\begin{align*}
			\Pr\{	\hat{N}_\alpha^{f_1}(t)=n\} &= \frac{\lambda^{2n}}{(2n)!}\sum_{k=0}^{\infty}\frac{(2n+k)!}{k!}  \frac{(-\lambda)^k}{\Gamma(\alpha(2n+k)+1)} \frac{\Gamma(bt+\alpha(2n+k))}{a^{\alpha(2n+k)}\Gamma(bt)} \\ & \ \  + \frac{\lambda ^{2n+1}}{(2n+1)!}\sum_{k=0}^{\infty}\frac{(2n+k+1)!}{k!}  \frac{(-\lambda)^k}{\Gamma(\alpha(2n+k+1)+1)}  \frac{\Gamma(bt+\alpha(2n+k+1))}{a^{\alpha(2n+k+1)}\Gamma(bt)}.
		\end{align*}
		Similarly, from \eqref{pp} and \eqref{mga}, we get the following form of pgf of  $\{	\hat{N}_\alpha^{f_1}(t)\}_{t \geq 0}$: 
		\begin{equation*}
			\hat{G}^{f_1}_{\alpha}(u,t)=\frac{\sqrt{u}+1}{2\sqrt{u}}\sum_{l=0}^{\infty}\frac{(-\lambda (1-\sqrt{u}))^l}{\Gamma(\alpha l+1)}\frac{\Gamma(bt+\alpha l)}{a^{\alpha l}\Gamma(bt)}+\frac{\sqrt{u}-1}{2\sqrt{u}}\sum_{l=0}^{\infty}\frac{(-\lambda (1+\sqrt{u}))^l}{\Gamma(\alpha l+1)}\frac{\Gamma(bt+\alpha l)}{a^{\alpha l}\Gamma(bt)}.
		\end{equation*}
		
	\end{example}
	\begin{example}
		
		($\{\hat{N}_\alpha(t)\}_{t \geq 0}$ time-changed by tempered stable subordinator)
		Let $\{D_{\theta,\nu}(t)\}_{t \ge 0}$ be a tempered stable subordinator whose associated  Bern\v{s}tein function  is
		\begin{equation*}
			f_2(s)=(\theta + s)^\nu - \theta^\nu,\ \  s>0,
		\end{equation*}
		where $0 < \nu < 1$ is the stability index and $\theta > 0$ is the tempering parameter.

		Its  $r$th moment is given by
		\begin{equation}\label{rtss}
			\mathbb{E}\big(D_{\theta,\nu}^r(t)\big)
			=
			e^{\theta^\nu t}
			\mathbb{E}\big(D_\nu^r(t)e^{-\theta D_\nu(t)}\big),\ \ r>0.
		\end{equation}
		On taking $f_2$ as a Bern\v stein function in \eqref{tc}, we get the following time-changed process:
		\begin{equation*}
			\hat{N}_\alpha^{f_2}(t)\coloneqq \hat{N}
			_\alpha (D_{\theta,\nu}(t)),\ \ t\ge0.
		\end{equation*}
		On using \eqref{rtss} in \eqref{pmftcl}, the pmf of $\{\hat{N}_\alpha^{f_2}(t)\}_{t \geq 0}$ is given by 
		{\small  \begin{align*}
				\Pr\{\hat{N}_\alpha^{f_2}(t)=n\} &= \frac{\lambda^{2n}}{(2n)!}\sum_{k=0}^{\infty}\frac{(2n+k)!}{k!}  \frac{e^{\theta^\nu t}(-\lambda)^k}{\Gamma(\alpha(2n+k)+1)}\mathbb{E}\big(D_\nu^{\alpha(2n+k)}(t)e^{-\theta D_\nu(t)}\big)
				\\ & \ \  + \frac{\lambda ^{2n+1}}{(2n+1)!}\sum_{k=0}^{\infty}\frac{(2n+k+1)!}{k!}  \frac{e^{\theta^\nu t}(-\lambda)^k}{\Gamma(\alpha(2n+k+1)+1)}\mathbb{E}\big(D_\nu^{\alpha(2n+k+1)}(t)e^{-\theta D_\nu(t)}\big).
		\end{align*}}
		Similarly, from \eqref{pp} we get the following pgf:
		\begin{align*}
			\hat{G}^{f_2}_{\alpha}(u,t)&=\frac{\sqrt{u}+1}{2\sqrt{u}}\sum_{l=0}^{\infty}\frac{(-\lambda (1-\sqrt{u}))^le^{\theta^\nu t}}{\Gamma(\alpha l+1)}
			\mathbb{E}\big(D_\nu^{\alpha l}(t)e^{-\theta D_\nu(t)}\big)\\&+\frac{\sqrt{u}-1}{2\sqrt{u}}\sum_{l=0}^{\infty}\frac{(-\lambda (1+\sqrt{u}))^le^{\theta^\nu t}}{\Gamma(\alpha l+1)}
			\mathbb{E}\big(D_\nu^{\alpha l}(t)e^{-\theta D_\nu(t)}\big).
		\end{align*}
	\end{example}
	
	\begin{example}[$\{\hat{N}_\alpha(t)\}_{t \geq 0}$ time-changed by inverse Gaussian subordinator]
		
		Let $\{Y(t)\}_{t \ge 0}$ be an inverse Gaussian subordinator with parameters $\delta >0$ and $\gamma >0$, and its associated  Bern\v{s}tein function is 
		\begin{equation*}
			f_3(s)=\delta\left(\sqrt{2s+\gamma^2}-\gamma\right),\ \ s>0.
		\end{equation*}

		\noindent Its $r$th moment is given by (see Kumar {\it et al.} (2011))
		\begin{equation}\label{migs}
			\mathbb{E}\!\left(Y^r(t)\right)
			=
			\sqrt{\frac{2}{\pi}}\,
			\delta
			\left(\frac{\delta}{\gamma}\right)^{r-\frac12}
			t^{\,r+\frac12}
			e^{\delta\gamma t}
			K_{\,r-\frac12}(\delta\gamma t),
			\ \  r \in \mathbb{R},
		\end{equation}
		where $K_\nu(x)$, $x>0$ is the modified Bessel function of the third kind 
		with real index $\nu$.
		
		The process $\{\hat{N}_\alpha(t)\}_{t \geq 0}$ time-changed by an independent inverse Gaussian subordinator  can be obtained by taking $f_3$ as a Bern\v stein function in \eqref{tc}. Thus,
		\begin{equation*}\label{fgau}
			\hat{N}_\alpha^{f_3}(t)\coloneqq \hat{N}
			_\alpha (Y(t)),\ \ t\ge0.
		\end{equation*}
		Using \eqref{pmftcl} and \eqref{migs}, its pmf is given by 
		{\small  \begin{align*}
				\Pr\{\hat{N}_\alpha^{f_3}(t)=n\} &=\sqrt{\frac{2\delta \gamma t}{\pi}}\frac{e^{\delta\gamma t}\lambda^{2n}}{(2n)!}\sum_{k=0}^{\infty}\frac{(2n+k)!}{k!}  \frac{(\delta \gamma^{-1}t)^{\alpha(2n+k)}(-\lambda)^k}{\Gamma(\alpha(2n+k)+1)} K_{\,\alpha(2n+k)-\frac{1}{2}}(\delta\gamma t)
				\\ & \ \  + \sqrt{\frac{2\delta \gamma t}{\pi}}\frac{e^{\delta\gamma t}\lambda^{2n+1}}{(2n+1)!}\sum_{k=0}^{\infty}\frac{(2n+k+1)!}{k!}  \frac{(\delta \gamma^{-1}t)^{\alpha(2n+k+1)}(-\lambda)^k}{\Gamma(\alpha(2n+k+1)+1)} K_{\alpha(2n+k+1)-\frac{1}{2}}(\delta\gamma t).
		\end{align*}}
		From \eqref{pp} and \eqref{migs}, its pgf is
		\begin{align*}
			\hat{G}^{f_3}_{\alpha}(u,t)&=\sqrt{\frac{2\delta \gamma t}{\pi}}e^{\delta\gamma t}\frac{\sqrt{u}+1}{2\sqrt{u}}\sum_{l=0}^{\infty}\frac{(-\lambda (1-\sqrt{u}))^l(\delta \gamma^{-1}t)^{\alpha l}}{\Gamma(\alpha l+1)}K_{\,\alpha l-\frac{1}{2}}(\delta\gamma t)\\&+\sqrt{\frac{2\delta \gamma t}{\pi}}e^{\delta\gamma t}\frac{\sqrt{u}-1}{2\sqrt{u}}\sum_{l=0}^{\infty}\frac{(-\lambda (1+\sqrt{u}))^l(\delta \gamma^{-1}t)^{\alpha l}}{\Gamma(\alpha l+1)}K_{\,\alpha l-\frac{1}{2}}(\delta\gamma t).
		\end{align*}
		
	\end{example}
	Next, we obtain a version of the law of iterated logarithm for $\{\hat{N}_\alpha^{f}(t)\}_{t \ge 0}$.
	\begin{theorem}\label{thm}
		Let the Bern\v stein function $f$ associated with the L\'evy subordinator $\{D_{f}(t)\}_{t\ge0}$ is regularly varying at $0+$ with index $0<\theta<1$, that is, $\lim_{x\rightarrow 0+}f(\lambda x)/f(x)=\lambda^\theta$, $\lambda>0$. Also, let \begin{equation*}
			g(t)=\frac{\log\log t}{\psi(t^{-1}\log\log t)},\ \ t>e,
		\end{equation*}
		where $\psi$ is the inverse of $f$. Then,
		\begin{equation*}
			\liminf_{t\rightarrow\infty}\frac{\hat{N}_{f}^{\alpha}(t)}{(g(t))^{\alpha}}\stackrel{d}{=}\frac{\lambda}{2}Y_{\alpha}(1)\theta^\alpha\Big(1-\theta\Big)^{\alpha(1-\theta)/\theta}.
		\end{equation*}	
	\end{theorem}
	\begin{proof}
		From \eqref{tc} and \eqref{sub}, we have $\hat{N}_{f}^{\alpha}(t)\overset{d}{=}\hat{N}(Y_{\alpha}(D_{f}(t)))$. Now, using the self-similarity property  of $\{Y_{\alpha}(t)\}_{t\ge0}$ given in \eqref{selfsimi}, we get 
		\begin{equation}\label{self}
			\hat{N}_{f}^{\alpha}(t)\overset{d}{=}\hat{N} (D^\alpha_f(t) Y_{\alpha}(1)).
		\end{equation}  
		Hence,
		\begin{align*}
			\liminf_{t\rightarrow\infty}\frac{\hat{N}_{f}^{\alpha}(t)}{(g(t))^{\alpha}}&\stackrel{d}{=}\liminf_{t\rightarrow\infty}\frac{\hat{N}(D^\alpha_f(t) Y_{\alpha}(1))}{(g(t))^{\alpha}}\\
			&=\liminf_{t\rightarrow\infty}\Big(\frac{\hat{N}(D^\alpha_f(t) Y_{\alpha}(1))}{D^\alpha_f(t) Y_{\alpha}(1)}\Big)\frac{D^\alpha_f(t)Y_{\alpha}(1)}{(g(t))^{\alpha}}\\
			&\stackrel{d}{=}\frac{\lambda}{2} Y_{\alpha}(1)\bigg(\liminf_{t\rightarrow\infty}\frac{D_{f}(t)}{g(t)}\bigg)^{\alpha},
		\end{align*}
		where in the last step we have used \eqref{limitS} and  the fact that $D_f(t)\to\infty$ as $t\to\infty$, a.s. (see Bertoin (1996), p. 73).  Finally, the proof follows from the law of the iterated logarithm of L\'evy subordinator (see Bertoin (1996),  p. 92, Theorem 14)
		\begin{equation*}\label{LIL}
			\liminf_{t\to\infty}\frac{D_{f}(t)}{g(t)}=\theta(1-\theta)^{(1-\theta)/\theta},\  \text{a.s.}
		\end{equation*}
	\end{proof}
	
	\begin{remark}
		From \eqref{self}, we have
		\begin{align*}
			\lim_{t \to \infty} \frac{\hat{N}^\alpha_f(t)}{t^\alpha}&\overset{d}{=} Y_\alpha(1) \lim_{t \to \infty} \frac{\hat{N} \big( D^\alpha_f(t) Y_\alpha(1) \big)}{D^\alpha_f(t) Y_\alpha(1)}  \frac{D^\alpha_f(t)}{t^\alpha} \\
			&\overset{d}{=} \frac{\lambda}{2}Y_\alpha(1) \lim_{t \to \infty} \Big( \frac{D_f(t)}{t} \Big)^\alpha, \ \  (\text{using \eqref{limitS}})\\
			&\overset{d}{=} \frac{\lambda}{2} Y_\alpha(1) \big( \mathbb{E} \big( D_f(1) \big) \big)^\alpha.
		\end{align*}
		Here, the last step follows from  the strong law of large numbers of a L\'evy subordinator (see Bertoin (1996), p. 92). Now using a contradiction argument as used in the proof of Proposition \ref{prop3.6}, it can be shown that the one-dimensional distributions of  $\{\hat{N}_\alpha^{f}(t)\}_{t \ge 0}$ are not infinitely divisible.
	\end{remark}
	
	\begin{remark}
		Note that the corresponding  Bern\v{s}tein functions $f_1(s)$, $f_2(s)$, and  $f_3(s)$, $s > 0$ of gamma subordinator, tempered stable subordinator, and  inverse Gaussian subordinator are regularly varying with index $\theta= 1$, $\theta= 0$, and $\theta= 1$, respectively (see Kumar {\it et al.} (2019), Remark 5.2). Therefore, the Theorem \ref{thm} can not be applied for these subordinators. However, the Bern\v{s}tein function associated with  stable subordinator is regularly varying with index $ 0< \theta={\alpha} <1 $. Thus, the Theorem \ref{thm} holds under the relaxation of the finite moments condition of  $\{D_{f}(t)\}_{t\ge0}$ in \eqref{tc}.
	\end{remark}
	
	\subsection{$\{\hat{N}_\alpha(t)\}_{t \geq 0}$ time-changed by inverse subordinator} Recall that the first passage time of a subordinator is called the inverse subordinator. We denote it  by  $\{Y_{f}(t)\}_{t\ge0}$ and note that $\mathbb{E}(Y^r_{f}(t))< \infty$, $r\ge 0$
	(see Aletti et al. (2018), Section 2.1). In this section, we consider the following time-changed process: 
	\begin{equation}\label{tc2}
		\bar{N}_\alpha^{f}(t)=\hat{N}_\alpha(Y_{f}(t)),\ \ t\ge0,
	\end{equation}
	where $\{\hat{N}_\alpha(t)\}_{t \geq 0}$ is independent with  inverse subordinator $\{Y_{f}(t)\}_{t\ge0}$.
	For $\alpha=1$, the process $\{\bar{N}^f_\alpha(t)\}_{t \geq 0}$ reduces to a time-changed version of $\{\hat{N}(t)\}_{t \geq 0}$, that is, 
	\begin{equation}\label{tc4}
		\bar{N}^{f}(t)=\hat{N}(Y_{f}(t)),\ \  t\ge0.
	\end{equation}
	The pmf $ \bar{p}_\alpha^f(n,t)=\Pr\{\bar{N}_\alpha^{f}(t)=n\},\ n \ge 0$ is given by 
	\begin{align*}
		\bar{p}_\alpha^f(n,t) &= \frac{\lambda^{2n}}{(2n)!}\sum_{k=0}^{\infty}\frac{(2n+k)!}{k!}  \frac{(-\lambda)^k}{\Gamma(\alpha(2n+k)+1)} \mathbb{E}\big(Y_f^{\alpha(2n+k)}(t)\big) \\ & \ \  + \frac{\lambda ^{2n+1}}{(2n+1)!}\sum_{k=0}^{\infty}\frac{(2n+1+k)!}{k!}  \frac{(-\lambda)^k}{\Gamma(\alpha(2n+1+k)+1)} \mathbb{E}\big(Y_f^{\alpha(2n+1+k)}(t)\big),
	\end{align*}  
	whose proof follows along the similar lines to that of Theorem \ref{5.1}.
	
	For $\alpha=1$, we get the pmf $\bar{p}^{f}(n,t)=\mathrm{Pr}\{\bar{N}^{f}(t)=n\}$ of $\{\bar{N}^{f}(t)\}_{t\ge0}$ as
	\begin{equation*}
		\bar{p}^{f}(n,t)=\frac{\lambda^{2n}}{(2n)!}\mathbb{E}\left(Y_{f}^{2n}(t)e^{-\lambda Y_{f}(t)}\right)+\frac{\lambda^{2n+1}}{(2n+1)!}\mathbb{E}\left(Y_{f}^{2n+1}(t)e^{-\lambda Y_{f}(t)}\right).				
	\end{equation*}
	For $|u|\le1$, its pgf  is given by
	\begin{equation*}
		\mathbb{E}(u^{\bar{N}_\alpha^{f}(t)}) =\frac{\sqrt{u}+1}{2\sqrt{u}}\sum_{l=0}^{\infty}\frac{(-\lambda (1-\sqrt{u}))^l}{\Gamma(\alpha l+1)}\mathbb{E}\big(Y^{\alpha l }_f(t)\big)+\frac{\sqrt{u}-1}{2\sqrt{u}}\sum_{l=0}^{\infty}\frac{(-\lambda (1+\sqrt{u}))^l}{\Gamma(\alpha l+1)}\mathbb{E}\big(Y^{\alpha l }_f(t)\big).
	\end{equation*}
	The proof of the above results follows similar lines to that of Theorem \ref{5.2}.
	\noindent The  mean of $\{\bar{N}_\alpha^{f}(t)\}_{t\ge 0}$ is given by 
	\begin{equation*}
		\mathbb{E}(\bar{N}_\alpha^{f}(t))=\frac{\partial}{\partial u}\big(\mathbb{E}(u^{\bar{N}_\alpha^{f}(t)}) \big)\Big|_{u=1}=-\frac{1}{4}+\frac{\lambda \mathbb{E}(Y^\alpha_f(t))}{2\Gamma(\alpha+1)}+\sum_{l=0}^{\infty} \frac{\mathbb{E}\big(Y^{\alpha l }_f(t)\big)(-2\lambda)^l}{4\Gamma(\alpha l+1)}.
	\end{equation*}
	
	Next, we prove that  the time-changed process $\{\bar{N}^f(t)\}_{t \ge 0}$ is a renewal process.

	\begin{theorem}\label{5.4}
		The time-changed process $\{\bar{N}^f(t)\}_{t \ge 0}$ is a renewal process with i.i.d. waiting times $\{J_n\}_{n \ge 1}$ that  satisfy 
		\begin{equation}\label{wait}
			\Pr\{J_n > t\}= \mathbb{E}\big((1+\lambda Y_{f}(t)) e^{-\lambda Y_{f}(t)}\big). 
		\end{equation}
		
	\end{theorem}

	\begin{proof} The proof follows similar lines to that of Theorem 2.2 and Theorem 4.1 of Meerschaert {\it et al.} (2011).
		Let $\{W_n\}_{n\ge 1}$ be an i.i.d. sequence of  Gamma$(\lambda,2)$ distributed  random variables, that is, 
		\begin{equation*}
			\Pr\{W_n > t\}= (1+\lambda t)e^{-\lambda t},
		\end{equation*}
		and $V_n = W_1+\cdots+W_n$. Then, $\hat{N}(t)\coloneqq \max\{n\ge 0 : V_n \le t\}$. Let
		\begin{equation*}\label{me}
			\tau_n = \sup\{ t > 0 : \bar{N}^f(t) < n\},
		\end{equation*}
		denote the jump times of  $\{\bar{N}^f(t)\}_{t \ge 0}$ .
		Note that $\{ \hat{N}(t)<n\}=\{V_n>t\}$. Then from the above expression, we have
		\begin{equation*}
			\tau_n =\sup\{t > 0 : Y_f(t) < V_n \}.
		\end{equation*}
		On applying  Lemma 2.1 of  Meerschaert {\it et al.} (2011), we get  $\tau_n = D_f (V_n-)$. Let us define the waiting times between jumps of the process $\{\bar{N}^f(t)\}_{t \ge 0}$ by $X_1 = \tau_1$ and 
		$X_n = \tau_n - \tau_{n-1}$ for $n \ge 2$. 
		Since $D_f(t)$ is a L\'evy process, it has no fixed points of discontinuity, that is, $\Delta D_f(t)=D_f(t)-D_f(t-)=0$ (a.s.) for a fixed $t\ge 0$ (see Applebum (2009), Lemma 2.3.2). So $D_f(t-) $ and $ D_f(t)$ are identically distributed for all $t \ge 0$.
		Now using a conditioning argument, we get
		\begin{align*}
			\mathbb{E}(e^{-s \tau_1}) = \mathbb{E}(e^{-s D_f(W_1-)})
			&= \mathbb{E} \big( \mathbb{E}(e^{-s D_f(W_1)} \mid W_1) \big)\\
			&= \mathbb{E}\big(e^{-W_1 f(s)}\big)
			= \frac{\lambda ^2}{(\lambda + f(s))^2},
		\end{align*}
		where in the last step we have used the Laplace transform of  Gamma($\lambda$, 2) distributed random variable.
		Also, note that (see Meerschaert {\it et al.} (2011), Eq. (4.6))
		\begin{equation}\label{9}
			\int_{0}^{\infty} e^{-st}\,\mathbb{E}(e^{-\lambda Y_f(t)})\mathrm{d}t
			= \frac{f(s)}{s\big(\lambda + f(s)\big)}.
		\end{equation}
		Let $y_{f}(x,t)$ be the pdf of  $\{Y_{f}(t)\}_{t\ge0}$. Then its Laplace transform is given by
		\begin{equation}\label{9.1}
			\int_{0}^{\infty} e^{-st} y_f(x,t)\mathrm{d}t
			= \frac{f(s)}{s}e^{-xf(s)}.
		\end{equation}

		Next, we evaluate the Laplace transform of the i.i.d. random variable $\{J_n\}_{n \ge 1}$ defined in \eqref{wait} as follows:
		\begin{align}\label{^}
			&\int_0^\infty e^{-st} \Pr\{J_n > t\}\mathrm{d}t = \int_0^\infty e^{-st} \mathbb{E}\big((1+\lambda Y_f(t)) e^{-\lambda Y_f(t)}\big)\mathrm{d}t \nonumber\\
			&= \int_0^\infty e^{-st} \mathbb{E}\big(e^{-\lambda Y_f(t)}\big)\mathrm{d}t+ \int_0^\infty e^{-st} \mathbb{E}\big(\lambda Y_f(t) e^{-\lambda Y_f(t)}\big)\mathrm{d}t \nonumber\\
			&= \frac{f(s)}{s(\lambda + f(s))}+\int_0^\infty e^{-st} \Big(\int_0^\infty \lambda x e^{-\lambda x} y_{f}(x,t)\mathrm{d}x\Big)\mathrm{d}t,\ \text{(using \eqref{9})} \nonumber\\
			&= \frac{f(s)}{s(\lambda + f(s))}+\int_0^\infty \lambda x e^{-\lambda x}\mathrm{d}x \Big(\int_0^\infty e^{-st} y_{f}(x,t)\mathrm{d}t\Big) \nonumber\\
			&= \frac{f(s)}{s(\lambda + f(s))}+\frac{ \lambda f(s)}{s}\int_0^\infty xe^{-(\lambda +f(s))x}\mathrm{d}x , \ \text{(using \eqref{9.1})} \nonumber\\
			&= \frac{f(s)}{s(\lambda + f(s))}+ \frac{ \lambda f(s)}{s(\lambda + f(s))^2}.
		\end{align}
		Consider $J_1 =T_1$ and 
		$J_n = T_n - T_{n-1}$ for $n \ge 2$,  and let $f_{J_n}$ be the pdf  of $\{J_n\}_{n \ge 1}$. Now using integration by parts, we get
		\begin{align*}
			\mathbb{E}(e^{-s T_1})&=\int_0^\infty e^{-st} f_{J_n}(t) \mathrm{d}t \nonumber \\
			&= s\int_0^\infty e^{-st}\big(1 - \Pr\{J_n > t\}\big)\, \mathrm{d}t \nonumber \\
			&= 1-\frac{f(s)}{(\lambda + f(s))}- \frac{ \lambda f(s)}{(\lambda + f(s))^2}, \ \text{(using \eqref{^})}\nonumber \\
			&=\frac{ \lambda ^2}{(\lambda + f(s))^2}=\mathbb{E}(e^{-s \tau_1}).
		\end{align*}
		Thus, by the uniqueness of Laplace transform we have  that $T_1 = J_1$ is equal in distribution to $\tau_1$. Now we can extend this argument as in the proof of Theorem 2.2 of Meerschaert {\it et al.} (2011) to establish that $(T_1,\ldots,T_n)$ is  identically distributed with  $(\tau_1,\ldots,\tau_n)$ for any positive integer  $n$. For example, when  $n=2$, we can write
		\begin{align*}
			\mathbb{E}\big(e^{-s_1 D_f(t_1)} e^{-s_2 D_f(t_1+t_2)}\big)
			= \mathbb{E}\big(e^{-(s_1+s_2)D_f(t_1)} e^{-s_2\big(D_f(t_1+t_2)-D_f(t_1)\big)}\big)
			= e^{-t_1 f(s_1+s_2)} e^{-t_2 f(s_2)},
		\end{align*}
		where we have used the fact that the L\'evy subordinator $D_f(t)$ has independent and stationary increments.
		Also,
		\begin{align*}
			\mathbb{E}(e^{-s_1 \tau_1 - s_2 \tau_2})
			&= \mathbb{E}\big(e^{-s_1 D_f(W_1-) - s_2 D_f((W_1+W_2)-)}\big)\\
			&= \mathbb{E}\Big( \mathbb{E}\left(e^{-s_1 D_{f}(W_1) - s_2 D_{f}(W_1 + W_2)}\,\middle|\, W_1, W_2\right)\Big)\\
			&= \mathbb{E}\Big( \mathbb{E}\Big(e^{-(s_1+s_2) D_{f}(W_1) - s_2 \big(D_{f}(W_1 + W_2)-D_f(W_1)\big)}\Big| W_1, W_2\Big)\Big)\\
			&= \mathbb{E}\Big( \mathbb{E}\left(e^{-(s_1+s_2) D_{f}(W_1) - s_2 D_{f}(W_2)}\,\middle|\, W_1, W_2\right)\Big)\\
			&= \mathbb{E}\Big(e^{-W_1 f(s_1+s_2)} e^{-W_2 f(s_2)} \Big)\\
			&= \frac{\lambda^2}{(\lambda + f(s_1+s_2))^2} \cdot \frac{\lambda ^2}{(\lambda + f(s_2))^2}.
		\end{align*}               Moreover, using the independence of $J_1$ and $J_2$, we have
		\begin{align*}
			\mathbb{E}(e^{-s_1 T_1} e^{-s_2 T_2})
			= \mathbb{E}(e^{-s_1 J_1} e^{-s_2 (J_1 + J_2)})
			= \frac{\lambda^2}{(\lambda + f(s_1+s_2))^2} \cdot \frac{\lambda ^2}{(\lambda + f(s_2))^2}= \mathbb{E}(e^{-s_1 \tau_1 - s_2 \tau_2}).
		\end{align*}
		Finally, using the continuous mapping theorem, we conclude that $(J_1,\ldots,J_n)$ identically distributed with $ (X_1,\ldots,X_n)$, where $\{X_n\}$ are the waiting times of $\{\bar{N}^f(t)\}_{t \ge 0}$.
	\end{proof}
	\begin{remark}
		From \eqref{tc4}, we have
		\begin{align*}
			\bar{p}^f(n,t)
			&= \int_{0}^{\infty} \Pr\{\hat{N}(x)=n \} y_f(x,t) \mathrm{d}x  \notag \\
			&= \int_{0}^{\infty} \Big(\frac{(\lambda x)^{2n}}{(2n)!}e^{-\lambda x}+\frac{(\lambda x)^{2n+1}}{(2n+1)!}e^{-\lambda x}\Big)
			y_f(x,t)\mathrm{d}x, \ \text{(using \eqref{pmf2pp})}.
		\end{align*}
		On taking Laplace transform on both sides, we get
		\begin{align}\label{mepr}
			\tilde{\bar{p}}^f(n,s)&=\int_{0}^{\infty} e^{-st}\bigg(\int_{0}^{\infty} \Big(\frac{(\lambda x)^{2n}}{(2n)!}e^{-\lambda x}+\frac{(\lambda x)^{2n+1}}{(2n+1)!}e^{-\lambda x}\Big)
			y_f(x,t)\mathrm{d}x\bigg)\mathrm{d}t \nonumber \\
			&=\int_{0}^{\infty} \Big(\frac{(\lambda x)^{2n}}{(2n)!}e^{-\lambda x}+\frac{(\lambda x)^{2n+1}}{(2n+1)!}e^{-\lambda x}\Big)
			\Big(\int_{0}^{\infty} e^{-st}y_f(x,t)\mathrm{d}t\Big)\mathrm{d}x \nonumber\\
			&=\frac{\lambda ^{2n} f(s)}{s(2n)!}\int _{0}^\infty x^{2n}e^{-(\lambda +f(s))x}\mathrm{d}x+\frac{\lambda ^{2n+1} f(s)}{s(2n+1)!}\int _{0}^\infty x^{2n+1}e^{-(\lambda +f(s))x}\mathrm{d}x \nonumber\\
			&= \Big(\frac{f(s)}{s(\lambda + f(s))}+ \frac{ \lambda f(s)}{s(\lambda + f(s))^2}\Big)\Big(\frac{\lambda^2}{(\lambda+f(s))^2}\Big)^n.
		\end{align}
		From \eqref{^}, we can observe that the first factor in \eqref{mepr} is the Laplace transform of $\Pr\{J_n > t\}$ and the second factor is the Laplace transform of $T_n = J_1 + J_2 + \cdots + J_n$ where $J_n$ are i.i.d. We denote the density of $T_n$ by $f_{J_n}^{*n}$, the $n$-fold convolution of the density function $f_{J_n}$ of $J_1$. Thus,
		\begin{align*}
			\tilde{\bar{p}}^f(n,s)&=\mathscr{L}(\Pr\{J_n > t\};s)\mathscr{L}(f_{J_n}^{*n}(t);s).
		\end{align*}
		On applying the inverse Laplace transform, we get    
		\begin{equation*}
			\bar{p}^f(n,t)=\int_{0}^{t}\Pr\{J_n > t-s\}f_{J_n}^{*n}(s)\mathrm{d}s,
		\end{equation*}
		which extends \eqref{fracpmf}.
	\end{remark}
	Next, we show that the process $\{\bar{N}_\alpha^{f}(t)\}_{t \ge 0}$ is a renewal process. For this purpose, we need the following result (see Maheshwari and Vellaisamy (2019), Lemma 4.1 and  Corollary 4.1): 
	\begin{lemma}\label{5.}
		Consider two independent inverse subordinators  $\{Y_{f_1}(t)\}_{t\ge 0}$ and $\{Y_{f_2}(t)\}_{t\ge 0}$ with associated  Bern\v{s}tein functions $f_1(s)$ and $f_2(s)$, respectively. Then
		\begin{equation}\label{mah}
			Y_{f_1}(Y_{f_2}(t)) \overset{d}{=} Y_{f_1 \circ f_2}(t), \ t \ge 0,
		\end{equation}
		where $(f_1 \circ f_2)(s) = f_1(f_2(s))$.
	\end{lemma}
	In particular, for $f_1(s) = s^{\alpha}$,  $0< \alpha < 1$ and $f_2(s) = f(s)$, the equation \eqref{mah} becomes
	\begin{equation*}
		Y_{\alpha}(Y_{f}(t)) \overset{d}{=} Y_{\psi}(t), \ t \ge 0,
	\end{equation*}
	where $\psi(s) = (f(s))^{\alpha}$.   
	
	\begin{proposition}
		The time-changed process $\{\bar{N}_\alpha^{f}(t)\}_{t \ge 0}$ is a renewal process with i.i.d. waiting times $\{J_n\}_{n \ge 1}$, which satisfy 
		\begin{equation}\label{90}
			\Pr\{J_n > t\}= \mathbb{E}\big((1+\lambda Y_{\psi}(t)) e^{-\lambda Y_{\psi}(t)}\big),   
		\end{equation} 
		where $\{Y_{\psi}(t)\}_{t \ge 0}$ is the inverse subordinator corresponding to ${\psi}(s)=(f(s))^\alpha$.
	\end{proposition}
	\begin{proof}
		From \eqref{tc2}, \eqref{sub} and the  Lemma \ref{5.}, we have
		\begin{equation*}
			\bar{N}^{f}_{\alpha}(t)
			=\hat{N}_{\alpha}(Y_{f}(t))
			\overset{d}{=} \hat{N}(Y_{\alpha}(Y_{f}(t)))
			\overset{d}{=} \hat{N}(Y_{\psi}(t))=\bar{N}^{\psi}(t),\ \ t \ge 0,    
		\end{equation*}
		where in the last step we have used \eqref{tc4}. Thus, the  process $\{\bar{N}_\alpha^{f}(t)\}_{t\ge 0}$ can be viewed as a  time-changed process, where $\{\hat{N}(t)\}_{t\ge 0}$ is time-changed 
		by an inverse subordinator $\{Y_{\psi}(t)\}_{t\ge 0}$ corresponding to the  Bern\v{s}tein function  $\psi(s) = (f(s))^{\alpha}$. Hence, from Theorem \ref{5.4} we deduce that  $\{\bar{N}_\alpha^{f}(t)\}_{t \ge 0}$ is a renewal process with i.i.d. waiting times $\{J_n\}_{n\ge 1}$ having distribution \eqref{90}. 
	\end{proof}
	Recall that in Theorem \ref{thm3.3}, we obtain the bivariate distribution of  $\{\hat{N}_{\alpha}(t)\}_{t\ge 0}$.  In the next result, we obtain the bivariate distributions of  $\{\bar{N}^{f}_{\alpha}(t)\}_{t\ge 0}$, which generalize  Theorem \ref{thm3.3}. Let $F(t)$ be the distribution function of i.i.d. waiting time $\{J_n\}_{n \ge 1}$ and $T_n=J_1+J_2+ \dots +J_n$ be the time of the $n$-th jump. Thus, ${\Pr}\{T_n \le t\}= F^{*n}(t)$. Consider  $\tau_n^{(k)} = T_{n+k} - T_n \overset{d}{=} T_k$, where $\tau_n^{(k)}$ represents the length of the time interval separating the $n$-th and $(n+k)$-th jump. Consequently,
	$${\Pr}\{\tau_n^{(1)} \in \mathrm{d}t\} = \mathrm{d}F(t)
	\ \text{and} \
	{\Pr}\{\tau_n^{(k)} \in \mathrm{d}t\} = \mathrm{d}F^{*k}(t),\ \ k \ge 1.$$
	
	\begin{theorem}\label{5.5}
		Let $0 < s\le t$  and let $\{Y_{\psi}(t)\}_{t\ge 0}$ be the inverse subordinator corresponding to the  Bern\v{s}tein function 
		$\psi(s) = (f(s))^{\alpha}$. Then for the nonnegative integers $0 <l \le n$,  the bivariate distributions of $\{\bar{N}_{\alpha}^{f}(t)\}_{t \ge 0}$ has the following form:
		\[
		{\Pr}\{ \bar{N}_{\alpha}^{f}(s) = l,\; \bar{N}_{\alpha}^{f}(t) = n \}
		=
		\begin{cases}
			\displaystyle
			\int_0^{s}
			\mathbb{E}\big((1+\lambda Y_{\psi}(t-u)) e^{-\lambda Y_{\psi}(t-u)}\big) \mathrm{d}F^{*l}(u),
			\ \ \text{if } n = l, \\[3ex]
			\displaystyle
			\int_0^{s} \mathrm{d}F^{*l}(u)
			\int_{s-u}^{t-u} \mathrm{d}F(v) \int_0^{t-(u+v)} 
			\mathbb{E}\big((1+\lambda Y_{\psi}( t - u - v - \xi ))\\ \ \ \times e^{-\lambda Y_{\psi}( t - u - v - \xi )}\big)
			\, \mathrm{d}F^{*(n-l-1)}(\xi),
			\ \ \text{if } n \ge l+1,
		\end{cases}
		\]
		where $\Pr\{J_n\le t\}$=F(t)= $ 1 - \mathbb{E}\big((1+\lambda Y_\psi(t)) e^{-\lambda Y_\psi(t)}\big)$ and $\mathrm{d}F^{*n}(t)$ is the $n$-fold convolution of $\mathrm{d}F(t)$, $n \ge 1$, with $\mathrm{d}F^{*0}(t) = \delta_0(t)$, the Dirac delta function at zero.
		\begin{figure}[htbp]
			\renewcommand{\figurename}{Fig.}
			\centering
			
			\unitlength=0.65mm
			\begin{picture}(50,30)
				\linethickness{1pt}
				\put(-70,10){\line(1,0){200}}
				\put(-70,13){\line(0,-1){6}}
				\put(-71.5,2.5){\small{$0$}}
				
				\put(-20,12){\line(0,-1){6}}
				\put(-20.5,2){\small{$s$}}
				
				\put(84,13){\line(0,-1){6}}
				\put(84.5,2.5){\small{$t$}}

				\put(-41,20){$\boldsymbol{\xleftrightarrow{\hspace{9.4cm}}}$}
				\put(22,23){\small{${\tau}_{l}^{(1)}$}}

				\put(-40,12){\line(0,-1){4}}
				\put(-43,-02){\small{${T}_{l}$}}
				

				
				\put(105,12){\line(0,-1){4}}
				\put(98,-2){\small{${T}_{l+1}$}}

			\end{picture}
			\caption{\tiny Waiting times between events for $n= l$.}\label{fig2}
		\end{figure}
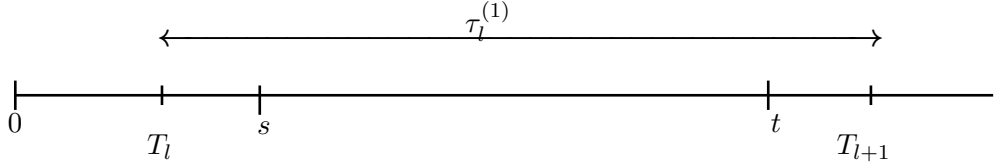
		
		\vspace{0.5cm}
		
		\begin{figure}[htbp]
			\renewcommand{\figurename}{Fig.}
			\centering
			
			\unitlength=0.65mm
			\begin{picture}(50,30)
				\linethickness{1pt}
				\put(-70,10){\line(1,0){200}}

				\put(-70,13){\line(0,-1){6}}
				\put(-71.5,2.5){\small{$0$}}
				
				\put(-30,12){\line(0,-1){6}}
				\put(-31.5,2){\small{$s$}}
				
				\put(94,13){\line(0,-1){6}}
				\put(94.5,2.5){\small{$t$}}

				\put(-50,25){$\boldsymbol{\xleftrightarrow{\hspace{2.4cm}}}$}
				\put(-30,29){\small{${\tau}_l^{(1)}$}}

				\put(-11,18){$\boldsymbol{\xleftrightarrow{\hspace{5.4cm}}}$}
				\put(22,23){\small{${\tau}_{l+1}^{(n-l-1)}$}}

				
				\put(75,25){$\boldsymbol{\xleftrightarrow{\hspace{2.4cm}}}$}
				\put(95,29){\small{${\tau}_n^{(1)}$}}
				
				\put(-50,12){\line(0,-1){4}}
				\put(-53,-02){\small{${T}_{l}$}}
				
				
				\put(-10,12){\line(0,-1){4}}
				\put(-13,-02){\small{${T}_{l+1}$}}
				
				\put(75,12){\line(0,-1){4}}
				\put(75,-2){\small{${T}_{n}$}}
				
				
				\put(115,12){\line(0,-1){4}}
				\put(108,-2){\small{${T}_{n+1}$}}

			\end{picture}
			\caption{\tiny Waiting times between events for $n\ge l+1$.}\label{fig3}
		\end{figure}

		\begin{proof}
			For $n = l$, we have (see Fig. \ref{fig2})
			\begin{align*}
				&{\Pr}\{\bar{N}_{\alpha}^{f}(s) = l,\; \bar{N}_{\alpha}^{f}(t) = l \}\\
				&= {\Pr}\{ 0 < T_l \le s, t< T_{l+1} \} \\
				&= {\Pr}\{ 0 < T_l \le s, \tau_l^{(1)} > t-T_{l} \} \\
				&= \int_0^{s} \mathrm{d}F^{*l}(u)\,
				{\Pr}\{ \tau_l^{(1)} > t - u \},
				\ \text{(since $T_l$ and $\tau_l^{(1)}$ are independent)},
			\end{align*}
			where $\tau_l^{(1)}  \overset{d}{=} T_1= J_1$.
			
			For $n \ge l+1$, since the interarrival times  are i.i.d., it follows from Fig. \ref{fig3} that
			{\small\begin{align*}
					{\Pr}&\{\bar{N}_{\alpha}^{f}(s) = l,\; \bar{N}_{\alpha}^{f}(t) = n \}\\
					&=\mathrm{Pr}\big\{0 < T_{l} \leq s, \,  s - T_{l} < {\tau}^{(1)}_{l} < t - T_{l}, \, 0 < \tau_{l+1}^{ (n-l-1)} < t - T_{l}-{\tau}^{(1)}_{l}, \, \\
					&\hspace{2cm}\tau^{(1)}_{n} > t - T_{l}-{\tau}^{(1)}_{l}-\tau_{l+1}^{ (n-l-1)}\big\}\\
					&= \int_0^{s} {\Pr}\{T_l \in \mathrm{d}u\}
					\int_{s-u}^{t-u} {\Pr}\{\tau_l^{(1)} \in \mathrm{d}v\} \int_0^{t-(u+v)} {\Pr}\{\tau_{l+1}^{(n-l-1)} \in \mathrm{d}\xi\}
					\int_{t-(u+v+\xi)}^{\infty}
					{\Pr}\{\tau_n^{(1)} \in \mathrm{d}\eta\}\\
					&= \int_0^{s} \mathrm{d}F^{*l}(u)
					\int_{s-u}^{t-u} \mathrm{d}F(v)
					\int_0^{t-(u+v)} \mathrm{d}F^{*(n-l-1)}(\xi) 
					{\Pr}\{ \tau_n^{(1)} > t - u - v - \xi \},
			\end{align*}}
			which completes the proof.
		\end{proof}
		
	\end{theorem}
	

\end{document}